\documentclass[a4paper]{amsart}
\usepackage[latin1]{inputenc}
\usepackage{amsfonts}
\usepackage{amsmath,latexsym,amssymb,amsfonts, amsthm,amsbsy}
\usepackage{esint}
\usepackage{cite}
\usepackage{graphicx}
\usepackage{amscd}
\usepackage{color}
\usepackage{bm}             
\usepackage{enumerate}
\usepackage[dvips]{epsfig}
\usepackage{psfrag}
\usepackage{epsfig}
\usepackage{mathrsfs}  
\usepackage{bm}

\usepackage[
  hmarginratio={1:1},     % equal left and right margins
  vmarginratio={1:1},     % equal top and bottom margins=
  textwidth=16cm,        % new text width
  textheight=21cm,
  heightrounded,          % always useful
]{geometry}

\usepackage{graphicx,color}
\usepackage[colorlinks]{hyperref}
\hypersetup{linkcolor=blue,citecolor=blue,filecolor=black,urlcolor=blue}

\newtheorem{theorem}{Theorem}
\newtheorem{corollary}[theorem]{Corollary}
\newtheorem{proposition}[theorem]{Proposition}
\newtheorem{lemma}[theorem]{Lemma}
\theoremstyle{definition}
\newtheorem{remark}{Remark}

\newtheorem*{conjecture}{Cartan-Hadamard conjecture in dimension $ \boldsymbol{n} $}

\numberwithin{theorem}{section}

\numberwithin{remark}{section}

\numberwithin{equation}{section}

\newcommand{\R}{\mathbb{R}}
\newcommand{\N}{\mathbb{N}}

\renewcommand{\div}{\,{\rm div}\,}

\newcommand{\dist}{{\rm dist}}

\newcommand{\eps}{\varepsilon}

\newcommand{\beq}{\begin{equation}}
\newcommand{\eeq}{\end{equation}}

\DeclareMathOperator{\loc}{loc}

\newcommand{\Mb}{{\mathbb{M}}}

\title[Rigidity results for Sobolev inequalities and related PDEs on CH manifolds]{Some rigidity results for Sobolev inequalities \\ and related PDEs on Cartan-Hadamard manifolds}

\author[M. Muratori and N. Soave]{Matteo Muratori and Nicola Soave}\thanks{}
\address{Matteo Muratori and Nicola Soave \newline \indent
	Dipartimento di Matematica,  Politecnico di Milano  \newline \indent
	Via Edoardo Bonardi 9, 20133 Milano, Italy}
\email{matteo.muratori@polimi.it; nicola.soave@polimi.it}

\makeatletter
\@namedef{subjclassname@2020}{%
  \textup{2020} Mathematics Subject Classification}
\makeatother


\begin{document}

\begin{abstract}
The Cartan-Hadamard conjecture states that, on every $n$-dimensional Cartan-Hadamard manifold $ \Mb^n $, the isoperimetric inequality holds with Euclidean optimal constant, and any set attaining equality is necessarily isometric to a Euclidean ball. This conjecture was settled, with positive answer, for $n \le 4$. It was also shown that its validity in dimension $n$ ensures that every $p$-Sobolev inequality ($ 1 < p < n $) holds on $ \Mb^n $ with Euclidean optimal constant. In this paper we address the problem of classifying all Cartan-Hadamard manifolds supporting an optimal function for the Sobolev inequality. We prove that, under the validity of the $n$-dimensional Cartan-Hadamard conjecture, the only such manifold %supporting optimal functions for a $p$-Sobolev inequality 
is $ \R^n $, and therefore any optimizer is an Aubin-Talenti profile (up to isometries). In particular, this is the case in dimension $n \le 4$. 

Optimal functions for the Sobolev inequality are weak solutions to the critical $p$-Laplace equation. Thus, in the second part of the paper, we address the classification of radial solutions (not necessarily optimizers) to such a PDE. Actually, we consider the more general critical or supercritical equation
\[
-\Delta_p u = u^q \, , \quad u>0 \, , \qquad \text{on } \Mb^n \, ,
\]
where $q \ge p^*-1$. We show that if there exists a radial finite-energy solution, then $\Mb^n$ is necessarily isometric to $\R^n$, $q=p^*-1$ and $u$ is an Aubin-Talenti profile. Furthermore, on model manifolds, we describe the asymptotic behavior of radial solutions not lying in the energy space $\dot{W}^{1,p}(\Mb^n)$, studying separately the $p$-stochastically complete and incomplete cases.
\end{abstract}

\maketitle

\section{Introduction and main results}
\normalcolor 
Given an integer $n \ge 2$ and $ p \in [1,n) $, it is well known that, on any $ n $-dimensional Cartan-Hadamard manifold $ \Mb^n $ (that is a complete and simply connected Riemannian manifold with nonpositive sectional curvature) the Sobolev inequality
\begin{equation}\label{generic-sobolev-gen}
\left\| f \right\|_{L^{p^\ast}(\Mb^n)}  \le C_{n,p} \left\| \nabla f \right\|_{L^p\left( \Mb^n \right)} \quad \forall f \in \dot{W}^{1,p}(\Mb^n) \, , \qquad p^\ast := \frac{np}{n-p}   \, ,
\end{equation}
holds, with a constant $C_{n,p}>0 $ that depends only on $n $ and $ p $ (see for example \cite[Lemma 8.1 and Theorem 8.3]{Hebey}). Here $ \dot{W}^{1,p}(\Mb^n) $ denotes the closure of $ C^1_c(\Mb^n) $ with respect to $  \| \nabla(\cdot) \|_{L^{p}(\Mb^n)} $. In particular, for $ p=1 $ we have
\begin{equation*}\label{sobolev-one}
\left\| f \right\|_{L^{\frac{n}{n-1}}\!(\Mb^n)}  \le C_{n,1} \left\| \nabla f \right\|_{L^1\left( \Mb^n \right)} \qquad \forall f \in \dot{W}^{1,1}(\Mb^n) \, ,
\end{equation*}
which by standard approximation arguments turns out to be equivalent to the \emph{isoperimetric inequality}
\begin{equation}\label{sobolev-isop}
\mathrm{Per}(\Omega) \ge \frac{1}{C_{n,1}} \, [V(\Omega)]^{\frac{n-1}{n}} \, ,
\end{equation}
where $ \Omega \subset \Mb^n $ is an arbitrary \emph{bounded} measurable set, $ dV $ stands for the volume measure on $\mathbb{M}^n$, and
$$
\mathrm{Per}(\Omega) := \sup \left\{  \int_{\Mb^n} \chi_{\Omega} \, \operatorname{div} \Phi \, dV : \ \Phi \in C^1_c(\Mb^n ; T\Mb^n) \, , \ \left\| \Phi \right\|_{L^\infty(\Mb^n)} \le 1 \right\} 
$$
is the perimeter function induced by $ dV $. 

\smallskip
A first nontrivial question concerns the exact value of the \emph{optimal constant} in \eqref{generic-sobolev-gen}, namely the smallest constant for which the inequality is true. That such constant must be larger than or equal to the Euclidean one is a standard fact due to the infinitesimally Euclidean structure of any (smooth) Riemannian manifold. Whether it is \emph{equal} to the latter is a much harder problem; the special case $ p=1 $ is known in the literature as the \emph{Cartan-Hadamard conjecture}, that we now recall. 

\begin{conjecture} \it
	Let $\mathbb{M}^n$ be an $n$-dimensional Cartan-Hadamard manifold. Then the \emph{Euclidean isoperimetric inequality} holds on $\mathbb{M}^n$, that is for every bounded measurable set $\Omega \subset \mathbb{M}^n$ it holds that
	\begin{equation}\label{ispo-ineq}
	\mathrm{Per}(\Omega) \ge n \,  \omega_n^{\frac1n} \, [V(\Omega)]^{\frac{n-1}{n}} \, ,
	\end{equation}
where $ \omega_n $ is the volume of the unit ball in $ \R^n $. Furthermore, equality holds if and only if $\Omega$ is isometric to a ball in $\R^n$ (up to a set of volume zero).
\end{conjecture}
So far, the conjecture has been settled, with positive answer, only up to dimension $4$ (see \cite{Beck,Weil,Croke,Kleiner}). Although we will discuss in more detail these issues in Section \ref{sec: optimal}, it is worth mentioning here that, as shown in \cite[Proposition 8.2]{Hebey}, if the optimal constant in \eqref{sobolev-isop} is Euclidean (namely the Cartan-Hadamard conjecture holds) then also the optimal constant in \eqref{generic-sobolev-gen} is Euclidean. 

\smallskip
A further related question concerns \emph{optimal functions}, that is, functions $  u \in \dot{W}^{1,p}(\Mb^n)$ attaining the optimal constant in \eqref{generic-sobolev-gen}. In the Euclidean space $\R^n$, it is well known since the celebrated results by Aubin \cite{Aubin} and Talenti \cite{Talenti} that such functions do exist, and have the explicit expression 
\begin{equation}\label{optimizer}
u(x) = a  \left( b +  \left|x-x_o\right|^\frac{p}{p-1}\right)^{-\frac{n-p}{p}} \qquad \text{for a.e.~} x \in \R^n
\end{equation}
for some $x_o \in \R^n$, $a \in \R \setminus \{ 0 \} $, and $ b>0 $. One of the main purposes of the present paper is to address the problem on a general Cartan-Hadamard manifold. More precisely, supposing that the Cartan-Hadamard conjecture holds in dimension $n$, we can completely characterize all the Cartan-Hadamard manifolds supporting an optimal function. In fact, up to isometries, the only possibility is that $\Mb^n = \R^n$. 

\begin{theorem}\label{rigidity-optimal}
	Let $\mathbb{M}^n$ be a Cartan-Hadamard manifold and $ 1<p<n $. Suppose that the Cartan-Hadamard conjecture in dimension $n$ holds. Let $ u \in \dot{W}^{1,p}(\mathbb{M}^n)$ be a nontrivial optimal function for the Sobolev inequality \eqref{generic-sobolev-gen}, in the sense that $ u \not \equiv 0 $ and 
	$$
	\frac{\left\|\nabla u \right\|_{L^p\left( \Mb^n \right)}}{\|u\|_{L^{p^*}\!\left( \Mb^n \right)}} = \inf_{f \in \dot{W}^{1,p}(\mathbb{M}^n), \, f \not \equiv 0 } \frac{\left\|\nabla f \right\|_{L^p\left( \Mb^n \right)}}{\| f \|_{L^{p^*}\!\left( \Mb^n \right)}}   \, .
	$$
	Then $\Mb^n$ is isometric to $\R^n$, and
	\beq\label{instant}
	u(x) = a  \left( b +  \mathrm{dist}\!\left(x,x_o\right)^\frac{p}{p-1}\right)^{-\frac{n-p}{p}} \qquad \text{for a.e.~} x \in \Mb^n
	\eeq
	for some $x_o \in \Mb^n$, $a \in \R \setminus \{ 0 \} $, and $ b>0 $, where $\mathrm{dist}\!\left(x,x_o\right)$ denotes the Riemannian distance of $x$ from $x_o$.
\end{theorem}

In particular, thanks to the validity of the Cartan-Hadamard conjecture in low dimension, we deduce the following.

\begin{corollary}
 Let $ n = 2,3,4 $, $ p \in (1,n) $, and let $\mathbb{M}^n$ be a  Cartan-Hadamard manifold. 	
	Suppose that there exists a nontrivial optimal function $ u \in \dot{W}^{1,p}(\mathbb{M}^n)$ for the Sobolev inequality \eqref{generic-sobolev-gen}. Then $\Mb^n$ is isometric to $\R^n$, and $u$ is of type \eqref{instant}.
\end{corollary}

\begin{remark}\label{rem: reg optimal sets}
In the above quoted papers where the Cartan-Hadamard conjecture was proved for $n \le 4$, the result is typically established for \emph{smooth} sets (actually submanifolds). That is, it is shown that 
	\begin{equation}\label{iso-smooth}
	\mathrm{Per}(\Omega) \ge n \,  \omega_n^{\frac1n} \, [V(\Omega)]^{\frac{n-1}{n}} 
	\end{equation}
for all bounded smooth sets $ \Omega \subset \Mb^n $, and moreover, if equality holds for some $ \Omega $ within this class, then $\Omega$ is isometric to a Euclidean ball. By approximation, it is straightforward to extend \eqref{iso-smooth} to general bounded measurable sets. It is less obvious that a bounded measurable set attaining equality in \eqref{iso-smooth} is smooth, and thus isometric to a Euclidean ball. Nevertheless, this is surely true up to dimension $n=7$, as a consequence of subsequent regularity results for the so-called \emph{isoperimetric hypersurfaces} (see for example \cite[Theorem 2]{GMT} or \cite[Corollary 3.7]{Morgan}). 
\end{remark}

{In proving} Theorem \ref{rigidity-optimal}, a key point lies in the fact that any (nonnegative) optimal function of the $p$-Sobolev inequality weakly solves, up to a multiplicative constant, the following \emph{critical $p$-Laplace equation}: 
\begin{equation}\label{equation-pl}
-\Delta_p u = u^{p^\ast-1}, \quad u > 0 \, , \qquad \text{on } \Mb^n  \, ,
\end{equation}
where we recall that $ \Delta_p u := \div\!\big( \left| \nabla u \right|^{p-2} \nabla u \big) $. A first crucial step consists in showing that such optimal functions are globally bounded and decay uniformly to zero at infinity. Since we do not assume curvature bounds on $ \Mb^n $ other than $ \mathrm{Sect} \le 0$ (where $ \mathrm{Sect} $ denotes the sectional curvature), it does not seem possible to obtain such properties by means of the usual Euclidean-like techniques. Therefore, we need to set up a specific argument that can also be extended to 
frameworks more general than Cartan-Hadamard manifolds; we refer to the proof of Proposition \ref{thm-vanishing} for the details. Once that decay and regularity of the solutions are proved, we are able to adapt the symmetrization technique, originally developed in \cite{Talenti}, on the manifold $\Mb^n$, obtaining the rigidity result of both the manifold $\Mb^n$ and the optimal function $u$. 

Recently, similar rigidity results regarding \emph{interpolation inequalities} were proved in the papers \cite{Kristaly-jmpa,Kristaly-potential}, either upon requiring or not the validity of the Cartan-Hadamard conjecture (see also \cite{Farkas}). As for the Sobolev inequality with $p=2$, it was shown in \cite[Theorem 1.1]{KM} that no radial optimal function can exist unless $ \Mb^n \equiv \R^n $, actually without assuming the Cartan-Hadamard conjecture. However, this result will now follow as a particular case of Theorem \ref{thm: rig} below. 

We also mention \cite{Le,PiVe,Xi}, which concern rigidity results for Sobolev inequalities on manifolds with \emph{nonnegative} or \emph{asymptotically nonnegative} curvature, that is, in the somehow complementary setting with respect to ours.

\smallskip
Having discussed the relation between optimizers of the Sobolev inequality and solutions to \eqref{equation-pl}, a further natural step consists in studying rigidity results regarding solutions to \eqref{equation-pl} that are not necessarily optimal functions. Again, in $\R^n$, the problem is essentially understood: if $p=2$, then the only solutions to \eqref{equation-pl} are of type \eqref{optimizer}, see \cite{CGS}; if $1<p<n$ with $p \neq 2$, then the same holds under the additional assumption that $u \in \dot{W}^{1,p}(\R^n)$, see \cite{DMMS, Sc, Ve} and the recent paper \cite{CFR} for a different original approach. In all these contributions, a key step consists in proving the radial symmetry of positive solutions. However, when working on a manifold, this step becomes particularly involved since powerful tools available in the Euclidean context, such as the moving planes method, do not work (with the exception of some particular cases for which we refer to \cite{AlDaGe}). Therefore, in the above generality the problem remains open, and we focus instead on radial solutions to \eqref{equation-pl}. 

In fact, we consider the more general \emph{critical} or \emph{supercritical equation}
\beq\label{LE}
-\Delta_p u = u^q \, , \quad u>0 \, ,  \qquad \text{on $\Mb^n$, with $q \ge p^\ast-1$} \, ,
\eeq
addressing existence and asymptotic properties of $W^{1,p}_{\loc}(\Mb^n) \cap L^\infty_{\loc}(\Mb^n)$ radial weak solutions (from now on we will simply write ``radial solution" for the sake of brevity). Given the radiality assumption, it is natural to suppose further that $\Mb^n$ is a Cartan-Hadamard \emph{model manifold}: namely, there exists a pole $o \in \Mb^n$ such that the metric is given, in polar (or spherical) global coordinates about $o$, by 
\beq\label{metric}
g \equiv dr^2 + \psi^2(r) \, g_{\mathbb{S}^{n-1}} \, , 
\eeq
where $r$ is the Riemannian distance of a point of coordinates $(r,\theta) \in \mathbb{R}^+ \times \mathbb{S}^{n-1} $ from $o$, $g_{\mathbb{S}^{n-1}}$ stands for the usual round metric on the unit sphere, and $\psi:[0,+\infty) \to [0,+\infty)$ is a regular function with $\psi(0) = 0$ and $\psi'(0) = 1$. The Cartan-Hadamard assumption turns out to be equivalent to the fact that $ \psi $ is in addition \emph{convex}. A prototypical example is represented by the choice $\psi(r) = \sinh r$, which gives rise to a well-known realization of the hyperbolic space $\mathbb{H}^n$. 

For notational convenience, from here on and without further mention, we set
\begin{equation}\label{def-theta}
\Theta(r) := \frac{\int_0^r \psi^{n-1} \, ds}{\psi^{n-1}(r)} \qquad \forall r > 0 \, ,
\end{equation}
that is $\Theta$ accounts for the volume-surface ratio of geodesic balls centered at the pole $o$. This function, as we will see below, takes a primary role in our radial results. 

\smallskip
When $p=2$, existence and qualitative properties of radial solutions to the \emph{Lane-Emden equation} \eqref{LE} on the hyperbolic space and on more general model manifolds was recently investigated in \cite{BFG, BGGV}, also for subcritical powers. In particular, in \cite[Proposition 2.1]{BFG} it is proved that when $p=2$ and $q \ge 2^*-1$ there exist infinitely many radial solutions to \eqref{LE}, under fairly general assumptions on $\psi$. Moreover, under stronger assumptions, the authors were able to completely characterize the asymptotic behavior of the solutions, see \cite[Theorem 2.4]{BFG}, showing also that radial solutions with \emph{finite energy} cannot exist. 
Here we generalize these results in two directions (provided $\psi$ is of Cartan-Hadamard type): on one hand, we weaken the asymptotic assumptions on $\psi$ which are needed in \cite[Theorem 2.4]{BFG} (see the discussion below Theorem \ref{thm: decay}); on the other hand, we extend the results to any $1<p<n$. 

\smallskip
In the sequel, if $u$ is a radial function with respect to a point $o \in \mathbb{M}^n$, that is $u(x) = \varphi(\dist(x,o))$ for some real function $\varphi$, for simplicity we adopt the notation $u \equiv u(r)$, with $ r \equiv r(x) := \dist(x,o) \in [0,+\infty)$.

\begin{theorem}\label{thm: rig}
Let $\Mb^n$ be a Cartan-Hadamard manifold, $1<p<n$ and $q\ge p^*-1$. Suppose that there exists a radial solution $u$ to \eqref{LE} such that
\beq\label{hp grad}
\int_{\Mb^n} \left|\nabla u\right|^p d{V} <+\infty \, .
\eeq
Then $\Mb^n$ is isometric to $\R^n$, $q=p^*-1$ and $u$ is of type \eqref{instant}.
\end{theorem}

Note that here we do not require $ \Mb^n $ to be a model manifold, although for convenience we will carry out complete proofs in that case only (see Remark \ref{rem: radial-non-model} below on the modifications needed so as to treat the general case). Moreover, we do not even require $u \in \dot{W}^{1,p}(\Mb^n)$, but only that the gradient is integrable. This, for instance, allows us to include solutions with \emph{positive limit at infinity}. It is worth mentioning that such solutions do exist, under suitable assumption on $\Mb^n$ (this was already observed in \cite{BFG}). In fact, in proving Theorem \ref{thm: rig} (for model manifolds) an interesting dichotomy arises according to different integrability properties of the function $\Theta$ defined through \eqref{def-theta}. More precisely, we have the following. 

\begin{theorem}\label{thm: sc vs si}
Let $\Mb^n$ be a Cartan-Hadamard model manifold, $1<p<n$ and $q \ge p^*-1$. Then there exist infinitely many radial solutions to \eqref{LE} and the following alternative occurs:
\begin{itemize}
\item[($i$)] If 
\beq\label{p-sc}
\Theta^\frac1{p-1} \not \in L^1(\R^+),
\eeq
then any such solution is decreasing and tends to $0$ as $r \to +\infty$.

\smallskip

\item[($ii$)] If instead
\beq\label{p-si}
\Theta^\frac1{p-1} \in L^1(\R^+),
\eeq
then any such solution is decreasing and tends to a positive constant as $r \to +\infty$.
\end{itemize}
\end{theorem}

To sum up, the critical or supercritical $p$-Laplace equation on a Cartan-Hadamard model manifold always admits infinitely-many solutions. Such solutions may vanish at infinity or not, according to the dichotomy entailed by \eqref{p-sc} and \eqref{p-si}. However, they never satisfy the integrability condition \eqref{hp grad}, unless $\Mb^n$ is isometric to $\R^n$, $q=p^*-1$, and $u$ is of type \eqref{instant}.

When $p=2$, assumption \eqref{p-sc} is equivalent to the \emph{stochastic completeness} of the model manifold at hand. This property is originally related to the fact that the trajectories of the Brownian motion acting on $ \Mb^n $, almost surely, do not blow up in finite time. In fact, such a property turns out to bear several analytic equivalent formulations, regarding both elliptic and parabolic equations (see \cite{Gryg,PRS,GIM}). When $ p \neq 2 $, it was already observed in \cite{MarVal,BPS} that \eqref{p-sc} can still be interpreted, at least from the point of view of elliptic PDEs, as a nonlinear version of stochastic completeness, to which we will refer as \emph{$p$-stochastic completeness} in analogy with the previous literature. %according to an analogous terminology adopted in those papers. 
Our Theorem \ref{thm: sc vs si} then connects the vanishing at infinity of radial solutions with this global property of the ambient model manifold $\Mb^n$.

Our last result concerns a more detailed study of the asymptotic behavior of radial solutions at infinity. 

\begin{theorem}\label{thm: decay}
Let $\Mb^n$ be a Cartan-Hadamard model manifold, $1<p<n$ and $q \ge p^* - 1 $. 
\begin{itemize}
\item[($i$)] Under assumption \eqref{p-sc}, suppose further that either there exists $\gamma \in [0,1)$ such that 
\beq\label{hp add 1}
\lim_{r \to +\infty}  \frac{r^\gamma \, \psi'(r)}{\psi(r)} =: \ell \in (0,+\infty) 
\eeq
or 
\beq\label{hp add 2}
\lim_{r \to +\infty}  \frac{\psi'(r)}{\psi(r)} = + \infty \qquad \text{and} \qquad \lim_{r \to +\infty} \frac{\psi(r)}{\psi'(r)} \left[ \log \left( \frac{\psi'(r)}{\psi(r)}\right) \right]' = 0  \, .
\eeq
If $u$ is a radial solution to \eqref{LE}, then 
\beq \label{hp add 3}
\lim_{r \to +\infty} \left( \int_0^r \Theta^\frac1{p-1}\,ds\right)^\frac{p-1}{q+1-p} u(r) =  \left(\frac{p-1}{q+1-p}\right)^\frac{p-1}{q+1-p} .
\eeq

\item[($ii$)] Under assumption \eqref{p-si}, if $u$ is a radial solution to \eqref{LE} with $\lambda:= \lim_{r \to +\infty} u(r)>0$, then
\[
\lim_{r \to +\infty} \left( \int_r^{+\infty} \Theta^\frac1{p-1} \, ds \right)^{-1}  \left( u(r)-\lambda \right) = \lambda^\frac{q}{p-1} \, .
\]
Moreover, the limit value $\lambda$ satisfies the universal bound
\beq\label{limit-lambda}
\lambda \le \left( \frac{p-1}{q+p-1} \right)^{\frac{p-1}{q+p-1}} \left(\int_0^{+\infty} \Theta^\frac{1}{p-1} \, ds \right)^{-\frac{p-1}{q+1-p}} \, .
\eeq
\end{itemize}
\end{theorem} 

Assumptions \eqref{hp add 1} and \eqref{hp add 2} entail growth conditions, which to some extent ensure that $\psi$ has at least an \emph{exponential-like} behavior at infinity (see below). On one hand, as already mentioned, the case $p=2$ in Theorems \ref{thm: rig}, \ref{thm: sc vs si} and \ref{thm: decay}-($i$) was partially covered by \cite[Proposition 2.1 and Theorem 2.4]{BFG}, for a class of model manifolds that is slightly more general than the Cartan-Hadamard one. On the other hand, besides the fact that we also consider the case $ p \neq 2 $, our assumptions include certain manifolds that were not covered therein. Indeed, in \cite{BFG} it is required that either \eqref{hp add 2} holds (but without the $ \psi/\psi' $ term in the rightmost limit) or that $\psi'(r)/\psi(r) \to \ell \in (0,+\infty)$ as $r \to+ \infty$, which is a particular case of \eqref{hp add 1}; this latter condition allows us to treat model functions of type
\beq\label{esempi-gamma}
\psi(r) \sim e^{c \, r^{1-\gamma}} \qquad \text{as $r \to +\infty$\,, with $\gamma \in (0,1)$ and $c>0$\,,}
\eeq
which do not fulfill the assumptions in \cite{BFG}. Note that these kinds of manifolds  have a relevant role both as concerns \emph{radial} Sobolev inequalities and nonlinear diffusion PDEs, as discussed in a series of recent papers \cite{MurBum,MRon,GMV}. On top of that, we stress that in Theorems \ref{thm: rig} and \ref{thm: sc vs si}-($i$) we do not need any additional assumption, whereas in \cite{BFG} similar results are obtained still under the aforementioned conditions on $\psi'/\psi$.

It is not difficult to check that \eqref{hp add 1} actually implies \eqref{p-sc}, whereas \eqref{hp add 2} in general does not (one can take for instance model functions as in \eqref{esempi-gamma} with $ \gamma = 1-p-\eps $ for $ \eps>0 $).

We point out that, without assuming \eqref{hp add 1} and \eqref{hp add 2}, the thesis of Theorem \ref{thm: decay} may fail. In fact, in the next proposition, we show that
if $\psi$ has a \emph{power-like} growth at infinity then the asymptotic behavior described in Theorem \ref{thm: decay} cannot hold. In addition, we can show the existence of Cartan-Hadamard model manifolds, whose function $ \psi $ does not have a power-like growth, where it is not even possible to describe precise asymptotics of solutions at infinity.

\begin{proposition}\label{prop: oscillazioni}
Let $\Mb^n$ be a Cartan-Hadamard model manifold, $1<p<n$ and $q \ge p^*-1$. Suppose that 
\beq\label{hp fail}
\liminf_{r \to +\infty} \frac{ \int_0^r \Theta^{\frac{p}{p-1}} \, \psi^{n-1} \, ds  }{ \psi^{n-1}(r) \, \Theta(r) \int_0^r \Theta^{\frac{1}{p-1}} \, ds } > 0 \, .
\eeq
Then \eqref{p-sc} holds, but formula \eqref{hp add 3} fails for radial solutions to \eqref{LE} . In particular, this is the case if \eqref{hp add 1} is satisfied with $ \gamma=1 $. 

Furthermore, for every $ 1<p<n $, $ q \ge p^\ast-1 $ and $\alpha>0$, one can construct a Cartan-Hadamard model manifold satisfying \eqref{p-sc},
\beq\label{hp fail exp}
\limsup_{r \to +\infty} \frac{\psi(r)}{e^{\ell r}} = + \infty \qquad \forall \ell>0 \, ,
\eeq
and a corresponding radial solution $u$ to \eqref{LE} with $ u(0)=\alpha $ such that the limit in \eqref{hp add 3} does not exist.
\end{proposition}

We finally mention that \emph{radial} and \emph{weighted} Euclidean equations, to some extent, can be related to the radial version of \eqref{LE} (see also problem \eqref{radial pb} below), by means of a change of variable introduced in \cite[Section 7]{GMV}. In this regard, at least when $p=2$, such problems have largely been studied previously: see for example \cite{cheng-li} for nonexistence results, \cite{NiKa,Mana,Wei-Ni} for existence of solutions that do not vanish at infinity and \cite{Ni-Yo,Yana,Yana-Yotsu} for the analysis of zeros and asymptotics of radial solutions. 

\begin{remark}\label{rem: confronto hp}
On one hand, the fact that the asymptotic behavior of solutions, in the power-like case, cannot be of type \eqref{hp add 3} is not surprising. Indeed, if $ \Mb^n \equiv \mathbb{R}^n $, apart from the well-known special case $ q=p^\ast-1 $, in \cite[Theorem 9.1]{QS} (for $ p=2 $ and $q>2^\ast-1$) it was established that the limit constant is different from the one appearing on the right-hand side of \eqref{hp add 3}. Moreover, for analogous {weighted} Euclidean equations (recall the above discussion), it was proved in \cite[Theorem 5.33]{DN} (for $p=2$ and $ q=2^\ast-1 $) that actually solutions tend to ``oscillate'' around the expected asymptotic behavior. On the other hand, a general condition valid for all $ p $ and $ q $ such as \eqref{hp fail} seemed to be unknown, as well as the fact that there are non-power-like model functions $ \psi $ for which the limit in \eqref{hp add 3} does not exist.
\end{remark}
 
 \subsection*{Structure of the paper} Section \ref{sec: optimal} is entirely devoted to the proof of Theorem \ref{rigidity-optimal}, along with crucial preliminary results dealing with a priori estimates for optimal functions. In Section \ref{sec: radial} we focus on radial solutions, proving Theorems \ref{thm: rig}, \ref{thm: sc vs si}, \ref{thm: decay} and Proposition \ref{prop: oscillazioni} after a series of technical lemmas.

\subsection*{Acknowledgments:} The authors are partially supported by the INdAM-GNAMPA group (Italy). The first author is also supported by the PRIN 2017 project ``Direct and Inverse Problems for Partial Differential Equations: Theoretical Aspects and Applications'' (Italy).

\section{Optimal functions for the $p$-Sobolev inequality}\label{sec: optimal}

The goal of this section is to establish that, upon assuming the validity of the \emph{Cartan-Hadamard conjecture}, optimal functions for the Sobolev inequality \eqref{generic-sobolev-gen} on a Cartan-Hadamard manifold $ \Mb^n $ cannot exist unless $ \Mb^n $ is (isometric to) $ \R^n $. 

By means of a classical variational argument, it is plain that any nonnegative optimal function for \eqref{generic-sobolev-gen} satisfies in a weak sense, up to a multiplication by a positive constant, the following $p$-Laplace equation:
\begin{equation}\label{equation-pl-apriori}
-\Delta_p u = u^{p^\ast-1}  \qquad \text{on } \Mb^n  \, .
\end{equation}
In the next two results, which are stated under more general assumptions and may have an independent interest, we establish global boundedness and vanishing at infinity for general nonnegative \emph{energy solutions} to \eqref{equation-pl-apriori}, namely nonnegative (weak) solutions that in addition belong to the Sobolev space $ \dot{W}^{1,p}(\Mb^n) $, such as optimal functions. As a consequence, we will in particular deduce that actually \eqref{equation-pl-apriori} implies \eqref{equation-pl}. It is worth mentioning that, without further bounds on the Ricci curvature of $\mathbb{M}^n$, it does not seem possible to derive gradient estimates, starting from the $L^\infty$ bounds, as in the Euclidean case. In general, also Calder\'on-Zygmund-type results may fail (see \cite{Pigola-survey}). Therefore, the fact that energy solutions to \eqref{equation-pl-apriori} decay at infinity does not follow from standard arguments, and we need to devise an ad hoc method which may be useful in different contexts.

\smallskip
First of all, we prove that energy solutions are globally bounded. Here and in the sequel, for the sake of readability, for all $ q \in [1,\infty] $ we set $ \| \cdot \|_q  := \| \cdot \|_{L^q(\Mb^n)} $.

\begin{lemma}\label{lemma-bk}
 Let $ 1<p<n $ and $\mathbb{M}^n$ be any complete, noncompact Riemannian manifold supporting the Sobolev inequality \eqref{generic-sobolev-gen}. Let $ u $ be an energy solution to \eqref{equation-pl-apriori}. Then $ u \in L^\infty(\Mb^n) $. 
\end{lemma}

\begin{proof}
The proof can be carried out exactly as in the Euclidean case, for which we refer to \cite[Lemma 2.1]{Ve} and \cite[Appendix E]{Pe}. We omit the details.
\end{proof}

We can then show that in fact solutions vanish at infinity. 

\begin{proposition}\label{thm-vanishing}
 Let $ 1<p<n $ and $\mathbb{M}^n$ be any complete, noncompact Riemannian manifold supporting the Sobolev inequality \eqref{generic-sobolev-gen}. Let $ u $ be an energy solution to \eqref{equation-pl-apriori}. 
	Then $u$ is of class $C^1(\Mb^n)$ and, given any $ o \in \Mb^n $, it holds
	\begin{equation*}\label{eq-vanish}
	\lim_{\dist(x,o)\to +\infty} u(x) = 0. 
	\end{equation*}
	\end{proposition}
	
In particular, this and the previous result hold on every Cartan-Hadamard manifold.
	
\begin{proof}
Note that \eqref{equation-pl-apriori}, written in local coordinates, is a weighted Euclidean $p$-Laplace equation, and hence one can apply for instance the regularity results of \cite[Theorem 1]{Tolks} (see also \cite{DiB}) to infer that $ u \in C^1(\Mb^n) $ (actually the gradient of $u$ is locally $\alpha$-H\"older continuous, for some $\alpha \in (0,1) $ that may vary from compact set to compact set. We refer also to \cite{MZ} for a more complete account of the literature concerning the regularity theory of $p$-Laplace-type equations).
	To prove the vanishing at infinity, we will exploit a localized version of the Moser iteration technique.	Let $ o_i \in \Mb^n $ be any sequence such that $ \lim_{i \to \infty} \dist(o_i,o) = +\infty $. Clearly, proving the thesis amounts to showing that
	\begin{equation}\label{k-lim}
	\lim_{i \to \infty} u(o_i) = 0 \, .
	\end{equation}
	Let $  \xi \in C^\infty([0,+\infty) )$ be a nonincreasing cut-off function satisfying
	$$
	\xi(r) = 1 \quad \text{for every } r \in [0,1] \, , \qquad \xi(r)=0 \quad \forall r \ge 2 \, , \qquad 0 \le \xi(r) \le 1 \quad \text{for every } r \in (1,2)  \, ,
	$$
	and consider the (decreasing) sequence of radii
	\begin{equation}\label{def-ri}
	R_{k+1} = \left[ 1 - \frac{1}{2(k+1)^2} \right] R_k \quad \forall k \in \N \, , \qquad R_0 = 2 \, .
	\end{equation}
	Note that
	$$
  	R_\infty := \lim_{k \to \infty} R_k \in (0,2),
	$$
	since 
	$$
	\prod_{k=0}^\infty \left[ 1 - \frac{1}{2(k+1)^2} \right] > 0 \, .
	$$
	We construct the following sequence of cut-off functions on $ \Mb^n $:
	$$
	\xi_k(x) := \xi^n\!\left(  \frac{\mathrm{dist}(x,o_i)-R_{k+1}}{R_k - R_{k+1}} + 1 \right) \qquad \forall x \in \Mb^n \, .
	$$
	Since each $ \xi_k $ is radial about $ o_i $ and the sequence $ R_k $ fulfills \eqref{def-ri}, the following estimates hold:
		\begin{equation}\label{est-grad-i}
		  \left| \nabla \xi_k^{\frac 1 p}(x) \right| \le  \frac{ 2n \, \| \xi' \|_\infty \left(k+1\right)^2}{p \, R_\infty} \chi_{ B_{R_k}(o_i) \setminus B_{R_{k+1}}(o_i) }(x)  \qquad \forall x \in \Mb^n \, .
		\end{equation}

	\begin{equation}\label{id-1}
	\begin{aligned}
 \frac{p^p(1+\alpha_k)}{(p+\alpha_k)^p} \int_{ \Mb^n }  \xi_k \left|  \nabla u^{ 1+\frac{\alpha_k}{p} } \right|^p   d{V}  = &\,  \int_{ \Mb^n } \xi_k \, u^{p^\ast+\alpha_k} \, d{V} \\
 & \, - \left( \frac{p}{p+\alpha_k} \right)^{p-1} \int_{ \Mb^n } u^{ 1+\frac{\alpha_k}{p} } \nabla \xi_k \cdot \nabla u^{1+\frac{\alpha_k}{p}} \left| \nabla u^{1+\frac{\alpha_k}{p}} \right|^{p-2}  d{V} \, .
  \end{aligned}
	\end{equation}
	Note that the second term on the right hand side can be bounded, using H\"older's and Young's inequalities, by 
	$$
	\begin{aligned}
	 & \left| \left( \frac{p}{p+\alpha_k} \right)^{p-1} \int_{ \Mb^n } u^{ 1+\frac{\alpha_k}{p} } \nabla \xi_k \cdot \nabla u^{1+\frac{\alpha_k}{p}} \left| \nabla u^{1+\frac{\alpha_k}{p}} \right|^{p-2}  d{V}  \right| \\
	   & \qquad = \,  \frac{p^p}{\left( p+\alpha_k \right)^{p-1}} \left|  \int_{ \Mb^n } u^{1+\frac{\alpha_k}{p}}  \nabla \xi_k^{\frac 1 p} \cdot \nabla u^{1+\frac{\alpha_k}{p}} \left| \nabla u^{1+\frac{\alpha_k}{p}} \right|^{p-2}\xi_k^{\frac {p-1}{p}} \, d{V} \right| \\
	& \qquad \le \, \frac{p^p}{\left( p+\alpha_k \right)^{p-1}} \left(  \int_{ \Mb^n } \left| \nabla \xi_k^{\frac 1 p} \right|^p  u^{p+ \alpha_k }   \, d{V} \right)^{\frac 1 p } \left(  \int_{ \Mb^n } \xi_k \left| \nabla u^{1+\frac{\alpha_k}{p}} \right|^p d{V}  \right)^{\frac {p-1} p } \\
	  & \qquad \le \, \left( \frac{p}{1+\alpha_k} \right)^{p-1} \int_{ \Mb^n } \left| \nabla \xi_k^{\frac 1 p} \right|^p  u^{p+ \alpha_k }   \, d{V} +  \frac{p-1}{p} \, \frac{p^p(1+\alpha_k)}{(p+\alpha_k)^p}  \int_{ \Mb^n }  \xi_k \left|  \nabla u^{1+\frac{\alpha_k}{p}} \right|^p   d{V} \, .
	\end{aligned}
	$$
	Therefore, from \eqref{id-1} we infer that
	\begin{equation}\label{id-1'}
	p^{p-1} \frac{1+\alpha_k}{(p+\alpha_k)^p} \int_{\Mb^n}  \xi_k \left|  \nabla u^{ 1+\frac{\alpha_k}{p} } \right|^p   d{V}  \le \left( \frac{p}{1+\alpha_k} \right)^{p-1} \int_{ \Mb^n } \left| \nabla \xi_k^{\frac 1 p} \right|^p  u^{p+ \alpha_k }   \, d{V} + \int_{\Mb^n} \xi_k \, u^{p^\ast+\alpha_k} \, d{V} \, .
	\end{equation}
	On the other hand, by convexity
	$$
	 \int_{ \Mb^n }  \left|  \nabla \!\left( \xi_k^{\frac 1 p} u^{1+\frac{\alpha_k}{p}} \right) \right|^p d{V} \le \left( \frac{p}{p-1} \right)^{p-1} \int_{ \Mb^n } \left| \nabla \xi_k^{\frac 1 p} \right|^p  u^{p+ \alpha_k }   \, d{V} + p^{p-1}  \int_{ \Mb^n }  \xi_k \left|  \nabla u^{1+\frac{\alpha_k}{p}} \right|^p   d{V} \, .
	$$
	This estimate, combined with \eqref{id-1'}, gives
	\begin{multline}\label{est-k1}
	\frac{1+\alpha_k}{(p + \alpha_k )^p} \int_{ \Mb^n }  \left|  \nabla \!\left( \xi_k^{\frac 1 p} u^{1+\frac{\alpha_k}{p}} \right) \right|^p   d{V}  \le \left\| u \right\|_\infty^{p^\ast-p} \int_{ \Mb^n } \xi_k \, u^{p+\alpha_k} \, d{V}  \\
	 +  \left(  \frac{p}{p-1} \right)^{p-1} \left[  \left( \frac{p-1}{1+\alpha_k} \right)^{p-1} + \frac{1+\alpha_k}{\left( p+\alpha_k \right)^p} \right] \int_{ \Mb^n } \left| \nabla \xi_k^{\frac 1 p} \right|^p  u^{p+ \alpha_k }   \, d{V} \, .
	\end{multline}
	Now, for $\alpha_0 \ge p^*-p$, we pick the sequence $ \alpha_k $ as follows: 
	 \begin{equation}\label{def-sigma-bis}
	\alpha_{k+1} = p^\ast - p + \frac{p^\ast}{p} \alpha_k \qquad \implies \qquad \alpha_k = \left(  \frac{p^\ast}{p} \right)^{k}\left( \alpha_0 + p \right) - p \, .
	\end{equation}
	From here on, for the sake of readability, we will let $ A $ denote a general positive constant which is independent of $  k $, but may depend on $ \alpha_0 $, $ n $, $p$, $ R_\infty $, $ \xi $, $ \|  u \|_\infty  $ and change from line to line. We also recall that $C_{n,p}$ denotes the constant of the Sobolev embedding in \eqref{generic-sobolev-gen}. Having that in mind, by virtue of \eqref{est-grad-i} and \eqref{def-sigma-bis} estimate \eqref{est-k1} entails 
	$$
	\begin{aligned}
     \frac{1}{C_{n,p}^p} \left( \int_{ B_{R_{k+1}}(o_i) } u^{ p +\alpha_{k+1} } \,  d{V} \right)^{\frac {p}{p^\ast}} &\le \,  \frac{1}{C_{n,p}^p} \left( \int_{ \Mb^n } \xi_k^{\frac{p^\ast}{p}} u^{p^\ast+\frac{p^\ast}{p} \alpha_k} \, d{V} \right)^{\frac {p}{p^\ast}}\\
      &\le  \, \int_{ \Mb^n }  \left|  \nabla \!\left( \xi_k^{\frac 1 p} u^{1+\frac{\alpha_k}{p}} \right) \right|^p   d{V} 
      \le  \, A^{k+1}  \left\| u \right\|_{ L^{p+\alpha_k}\left( B_{R_k}(o_i) \right) }^{p+\alpha_k} ,
     \end{aligned}
	$$
	namely 
	$$
	\left\| u \right\|_{ L^{p+\alpha_{k+1}}\left( B_{R_{k+1}}(o_i) \right) } \le A^{\frac{k+1}{p + \alpha_k}} \left\| u \right\|_{ L^{p+\alpha_k}\left( B_{R_k}(o_i) \right) } \le  A^{ \sum_{h=0}^{k} \frac{h+1}{p+\alpha_h}} \left\| u \right\|_{ L^{p+\alpha_0}\left( B_{2}(o_i) \right) }   ,
	$$
	so that by letting $ k \to \infty $ we end up with $\left\| u \right\|_{ L^{\infty}\left( B_{R_{\infty}}(o_i) \right) } \le A \left\| u \right\|_{ L^{p+\alpha_0}\left( B_{2}(o_i) \right) }  $,
	which in turn yields \eqref{k-lim} upon letting $ i \to \infty $, since $ u \in L^{p+\alpha_0}(\Mb^n) $.
	\end{proof}
	
The strategy of proof of Theorem \ref{rigidity-optimal} is a suitable combination of the celebrated symmetrization tools introduced in \cite[Lemma 1]{Talenti} (see also the simultaneous paper \cite{Aubin}) and adapted to the manifold setting in \cite[Proposition 8.2]{Hebey}. There, as mentioned in the Introduction, the author proves that the validity of the Cartan-Hadamard conjecture ensures that the optimal constant in \eqref{generic-sobolev-gen} is indeed Euclidean for every Cartan-Hadamard manifold, but existence/nonexistence of optimal functions is not investigated. Before recalling some basics of the radial symmetrization technique, we point out that for our strategy to work it is crucial that $u$ and its superlevel sets are bounded, which is guaranteed by Lemma \ref{lemma-bk} and Proposition \ref{thm-vanishing}.

Given a measurable function $ f : \Mb^n \to \mathbb{R}^+ $ such that 
$$
V \! \left( \left\{ x \in \Mb^n : \, f(x) > t \right\} \right) < +\infty \qquad \forall t>0 \, , $$
we can introduce its \emph{Euclidean radially decreasing rearrangement} $ f^\star : \R^n \to \R^+ $ by setting 
$$
f^\star(y) := \int_0^{+\infty} \chi_{ \left\{ x \in \Mb^n : \, f(x) > t  \right\}^\star }(y) \, dt \qquad \forall y \in \mathbb{R}^n \, ,
$$
where, for every measurable set $ A \subset \Mb^n $ of finite volume, $ A^\star \subset \R^n $ denotes the Euclidean ball centered at the origin having the same (Euclidean) volume as $ A $, namely $ {V}(A)=|A^\star| $. By construction $ f^\star  $ is a (measurable) function that depends only on the variable $ |y| $. and is nonincreasing with respect to it. With some abuse of notation, for the sake of readability, below we will sometimes write $ f^\star(|y|) $. Since the superlevel sets of $ f^\star $ have the same Lebesgue measure as the Riemannian volume measure of the corresponding superlevel sets of $f$, thanks to the classical layer-cake representation (see for example \cite[Theorem 1.13]{LL}) the two functions also share $ L^q $ norms:
\begin{equation*}\label{layer-cake}
\int_{\mathbb{R}^n} \left( f^\star \right)^q dy = \int_{\Mb^n} f^q \, d {V} \qquad \forall q \in [1,\infty) \, .
\end{equation*}

The last key ingredient we need is the so-called \emph{coarea formula}. This is a well-established result originally due to Federer \cite[Theorem 3.1]{Fed}, and later extended to merely $ W^{1,1}_{\mathrm{loc}} $ functions, up to choosing a precise representative (we refer to \cite{MSZ} and the literature quoted therein). 
\begin{proposition}\label{coarea}
Let $ f : \Mb^n \to \R $ be a locally Lipschitz function and $ g : \Mb^n \to \R^+ $ a measurable function. Then it holds
\begin{equation}\label{formula-coarea}
\int_{\Mb^n} g \left| \nabla f \right| dV = \int_{\R} \int_{f^{-1}(\{s\})} g \, d\sigma \, ds \, ,
\end{equation} 
where $ d\sigma $ stands for the $(n-1) $-dimensional Hausdorff measure induced by $ dV $.
\end{proposition}

We are now in position to prove the main result of this section. 

\begin{proof}[Proof of Theorem \ref{rigidity-optimal}]
With no loss of generality, we can and will assume that $ u $ is nonnegative, thanks to the plain fact that if $ u $ is an optimal function, so is $ |u| $. Let us then introduce the volume function 
$$
\mathsf{V}(t) := V \! \left( \left\{ x \in \Mb^n : \, u(x) > t \right\} \right) \qquad \forall t>0 \, . 
$$
Clearly $ \mathsf{V}(t) $ is finite for all $t>0 $ since $ u \in L^{p^\ast}(\Mb^n) $. Moreover, by definition, it is a nonincreasing function, thus its pointwise derivative $ \mathsf{V}'(t) $ exists, is finite and nonpositive for almost every $ t>0 $. 

As observed above, we know that $u$ is a nonnegative energy solution to the Euler-Lagrange equation \eqref{equation-pl-apriori}, up to a multiplication by a constant. By virtue of Lemma \ref{lemma-bk} and Proposition \ref{thm-vanishing}, we can assert that it is bounded, of class $C^1(\Mb^n)$, and vanishes at infinity. In addition, by the strong maximum principle (see for example \cite[Proposition 6.4]{PRS}), we have that $u$ is strictly positive on the whole $\Mb^n$, so that it complies with \eqref{equation-pl}.

By regularity, the coarea formula \eqref{formula-coarea} holds with $ f = u $, so that upon choosing $ g $ as the characteristic function of each superlevel set $ \left\{ x \in \Mb^n : \, u(x) > t \right\}  $ we end up with the identity  
\begin{equation*}\label{formula-coarea-1}
\int_{u^{-1}\left((t,+\infty)\right)} \left| \nabla u \right| dV = \int_t^{+\infty} \sigma\!\left(u^{-1}(\{s\})\right) ds \qquad \forall t>0 \, ,
\end{equation*} 
which yields 
\begin{equation}\label{formula-coarea-2}
\frac{d}{dt} \int_{u^{-1}\left((t,+\infty)\right)} \left| \nabla u \right| dV = - \sigma\!\left(u^{-1}(\{t\})\right) \qquad \text{for a.e.~} t>0 \, .
\end{equation} 
Similarly, by using the same function $g$ multiplied by $ \left| \nabla u \right|^{p-1} $ we obtain 
\begin{equation}\label{formula-coarea-2-bis}
\frac{d}{dt} \int_{u^{-1}\left((t,+\infty)\right)} \left| \nabla u \right|^{p} dV = - \int_{u^{-1}\left( \{ t \} \right)} \left| \nabla u \right|^{p-1} d\sigma \qquad \text{for a.e.~} t>0 \, .
\end{equation} 
Note that the coarea formula itself guarantees that 
\begin{equation}\label{formula-coarea-3}
0 < \int_{u^{-1}\left( \{ t \} \right)} \left| \nabla u \right|^{p-1} d\sigma  < +\infty \qquad \text{for a.e.~} t \in \left(0, \left\| u \right\|_\infty \right) .
\end{equation} 
On one hand, this is easily seen by testing it with the characteristic function of the set of critical points $  \left\{ x \in \Mb^n : \, \left| \nabla u(x) \right| = 0 \right\} $, which makes sure that for a.e.~$ t>0 $ the function $ x \mapsto \left| \nabla u (x) \right| $ is $ \sigma $-a.e.~positive on $ u^{-1}(\{t\}) $, and $ \sigma\!\left(u^{-1}(\{t\})\right)>0 $ for every $ t $ as in \eqref{formula-coarea-3} by virtue of the continuity of $u$. On the other hand, finiteness follows from \eqref{formula-coarea-2-bis}.

The derivative in \eqref{formula-coarea-2} can be bounded from below by resorting to H\"older's inequality. Indeed, for incremental ratios we have (for all $ t,h>0 $)
\begin{multline*}
 \frac{ \int_{u^{-1}\left((t,+\infty)\right)} \left| \nabla u \right| dV - \int_{u^{-1}\left((t+h,+\infty)\right)} \left| \nabla u \right| dV }{h} \\
\le  \left( \frac{ \int_{u^{-1}\left((t,+\infty)\right)} \left| \nabla u \right|^p dV - \int_{u^{-1}\left((t+h,+\infty)\right)} \left| \nabla u \right|^p dV }{h} \right)^{\frac 1 p} \left( \frac{\mathsf{V}(t)-\mathsf{V}(t+h)}{h} \right)^{\frac{p-1}{p}} ,
\end{multline*}
so that, by passing to the limit as $ h \to 0 $, and using \eqref{formula-coarea-2}--\eqref{formula-coarea-2-bis}, we deduce the bound
\begin{equation}\label{coarea-key}
\sigma\!\left(u^{-1}(\{t\})\right) \le \left( \int_{u^{-1}\left( \{ t \} \right)} \left| \nabla u \right|^{p-1} d\sigma \right)^{\frac{1}{p}} \left| \mathsf{V}'(t) \right|^{\frac{p-1}{p}} \qquad \text{for a.e.~} t > 0 \, .
\end{equation}
Note that estimate \eqref{coarea-key} itself, combined with \eqref{formula-coarea-3}, ensures that $ \mathsf{V}'(t) $ is nonzero for almost every $ t \in (0,\| u \|_\infty) $, whence 
\begin{equation}\label{coarea-key-bis}
\frac{\sigma^p\!\left(u^{-1}(\{t\})\right)}{\left| \mathsf{V}'(t) \right|^{p-1}} \le \int_{u^{-1}\left( \{ t \} \right)} \left| \nabla u \right|^{p-1} d\sigma \qquad \text{for a.e.~} t \in \left(0, \left\| u \right\|_\infty \right) .
\end{equation}
If the Cartan-Hadamard conjecture in dimension $n$ holds, then
\begin{equation}\label{ch-conj-1} 
n \, \omega_n^{\frac 1 n} \, \mathsf{V}^{\frac{n-1}{n}}(t) \le \mathrm{Per}\!\left( u^{-1}\!\left((t,+\infty) \right) \right) \le \sigma\!\left(u^{-1}(\{t\})\right) \qquad \forall t>0 \, .
\end{equation}
(Recall that $ \mathrm{Per}(\Omega) = \sigma(\partial\Omega) $ provided $ \Omega $ is smooth enough, while in general we have that $ \mathrm{Per}(\Omega) \le \sigma(\partial\Omega)$ as consequence of the structure theorem for sets with finite perimeter, see for example \cite{AFP, Mag}). 

It is worth observing that in \eqref{ch-conj-1} we have implicitly exploited two additional key properties of $u$: continuity (so that $ \partial u^{-1}((t,+\infty)) \subseteq u^{-1}(\{t\}) $), and boundedness of the superlevel sets (which allows us to apply \eqref{ispo-ineq}), ensured by Proposition \ref{thm-vanishing}. As a consequence, integrating \eqref{coarea-key-bis} along with \eqref{formula-coarea-2-bis} entails 
\begin{equation}\label{ch-conj-2}
\int_0^{\left\| u \right\|_\infty} n^p \, \omega_n^{\frac p n} \, \frac{\mathsf{V}^{\frac{n-1}{n}p}(t)}{\left| \mathsf{V}'(t) \right|^{p-1}} \, dt \le \int_0^{\left\| u \right\|_\infty} \frac{\sigma^p\!\left(u^{-1}(\{t\})\right)}{\left| \mathsf{V}'(t) \right|^{p-1}} \, dt \le  \int_{u^{-1}\left((0,\| u \|_\infty) \right)} \left| \nabla u \right|^{p} dV = \int_{\Mb^n} \left| \nabla u \right|^p dV  \, .
\end{equation}
Let us consider the Euclidean radially decreasing rearrangement $u^\star $ of $ u $, which by definition shares with $ u $ the same volume function $ \mathsf{V}(t) $ and the same $ L^\infty $ norm. It is straightforward to check that it is continuous (otherwise $ \mathsf{V}'(t) $ would vanish in an interval). We claim that it is locally Lipschitz. Indeed, upon integrating \eqref{formula-coarea-2} and using \eqref{sobolev-isop} (here the optimal value of the isoperimetric constant is inessential), for all $t,h>0 $ we have: 
\begin{equation}\label{ch-lip-radial-1}
\begin{aligned}
L_t \left( \mathsf{V}(t)-\mathsf{V}(t+h) \right) &\ge \int_{u^{-1}\left((t,t+h]\right)} \left| \nabla u \right| dV \\ &= \int_t^{t+h} \sigma\!\left(u^{-1}(\{s\})\right) ds
 \ge \frac{ \int_{t}^{t+h} \mathsf{V}^{\frac{n-1}{n}}(s) \, ds}{C_{n,1}}  \ge \frac{\mathsf{V}^{\frac{n-1}{n}}(t+h)}{C_{n,1}} \, h \, ,
\end{aligned}
\end{equation}
where $ L_t $ stands for the Lipschitz constant of $ u $ in $ u^{-1}\!\left( (t,+\infty) \right) $. Given any $ r_2>r_1>0 $ complying with $ 0 < u^\star(r_2)< u^\star(r_1) \le \| u \|_\infty $, from the definition of $ \mathsf{V}(t) $ it is clear that $ \mathsf{V}(u^\star(r_2)) \le \omega_n \, r_2^n $ and $ \mathsf{V}(u^\star(r_1)-\eps) \ge \omega_n \, r_1^n $ for arbitrarily small $ \eps>0 $, so that by putting $ t = u^\star(r_2) $ and $ t+h = u^\star(r_1)-\eps $ in \eqref{ch-lip-radial-1} we end up with 
\begin{equation*}\label{ch-lip-radial-2}
L_{u^\star(r_2)} \, \omega_n \left( r_2^n-r_1^n \right) \ge \frac{\omega_n^{\frac{n-1}{n}}}{C_{n,1}} \, r_1^{n-1}\left( u^\star(r_1)-\eps - u^\star(r_2) \right) ,
\end{equation*}
and this readily implies, upon letting $ \eps \to 0$ and $ r_1 \to r_2^- $, that $ u^\star $ is Lipschitz in the open set $ \left\{ x \in \R^n: \, u^\star(x)>u^\star(r_2) \right\} $ with constant $ L_{u^\star(r_2)} \, C_{n,1} \, n \, \omega_n^{ 1 / n} $. Because $ u^\star>0 $ everywhere, the claim follows.

At this stage, given the local Lipschitz regularity of $ u^\star $ (recall Proposition \ref{coarea}), we can repeat all the above computations with $ u $ replaced by $ u^\star $ (and $ \Mb^n $ by $ \R^n $).
Since $ \nabla u^\star $ is constant on every level set $ \left( u^\star \right)^{-1}\! \left( \{ t \} \right)  $, and the latter is the boundary of the Euclidean ball $ \left( u^\star \right)^{-1}\!( (t,+\infty) ) $, for almost every $  t \in (0,\| u \|_\infty) $ (where $ \exists \mathsf{V}'(t)<0 $) both \eqref{coarea-key-bis} and \eqref{ch-conj-1}, thus \eqref{ch-conj-2}, hold as identities, whence
\begin{equation*}\label{ch-conj-rad}
\int_0^{\left\| u \right\|_\infty} n^p \, \omega_n^{\frac p n} \, \frac{\mathsf{V}^{\frac{n-1}{n}p}(t)}{\left| \mathsf{V}'(t) \right|^{p-1}} \, dt = \int_{\R^n} \left| \nabla u^\star \right|^p dV  \, .
\end{equation*}
In particular, from \eqref{ch-conj-2}, we deduce the P\'olya-Szeg\H{o}-type inequality $ \| \nabla u^\star \|_{L^p (\R^n)} \le \| \nabla u \|_{L^p(\Mb^n)} $. However, because $ \| u^\star \|_{L^{p^\ast}(\R^n)} =  \| u \|_{L^{p^\ast}(\Mb^n)}$ and the optimal constant in \eqref{generic-sobolev-gen} is not smaller than the Euclidean one, the only possibility is that $ u^\star $ is also an optimal function for the $p$-Sobolev inequality in $ \R^n $, and therefore \eqref{ch-conj-2} is actually an identity. In particular, it holds 
$$
\int_0^{\left\| u \right\|_\infty} \frac{\sigma^p\!\left(u^{-1}(\{t\})\right) - n^p \, \omega_n^{\frac p n} \, \mathsf{V}^{\frac{n-1}{n}p}(t) }{\left| \mathsf{V}'(t) \right|^{p-1}} \, dt  = 0 \, ,
$$
which in view of \eqref{ch-conj-1} yields
$$
\sigma\!\left(u^{-1}(\{t\})\right) = \mathrm{Per}\!\left( u^{-1}\!\left((t,+\infty) \right) \right)  = n \, \omega_n^{\frac 1 n} \, \mathsf{V}^{\frac{n-1}{n}}(t) \qquad \text{for a.e.~} t \in (0,\| u \|_\infty) \, .
$$
Thanks to the rigidity result encompassed by the Cartan-Hadamard conjecture, this implies that almost every superlevel set $ A_t := u^{-1}((t,+\infty) )$ is isometric to a Euclidean ball of volume $ \mathsf{V}(t) $, up to a set of volume zero. In particular, we can deduce that $ \overline{A}_t $ is isometric (in the metric sense) to a closed Euclidean ball. Hence, since $ u $ is continuous, has no zeros and $ \Mb^n $ is noncompact, there exist a decreasing sequence $ t_k \to 0 $ and a corresponding increasing sequence $ R_k \to +\infty $ such that 
$$
\Mb^n = \bigcup_{k=0}^{\infty}  {\overline{A}_{t_k}} \, , \qquad \overline{A}_{t_k} \Subset {A}_{t_{k+1}} \quad \forall k \in \N \, ,
$$
where each $ {\overline{A}_{t_k}} $ is isometric to $ \overline{B}^e_{R_k} $, the latter symbol denoting the closed Euclidean ball of radius $ R_k $ centered at the origin.  This means that for all $ k \in \N $ one can find a bijective map $ T_k : \overline{B}^e_{R_k} \to {\overline{A}_{t_k}} $, along with its inverse $ S_k:  {\overline{A}_{t_k}}  \to \overline{B}^e_{R_k} $, such that 
\beq \label{iso-1}
\mathrm{dist}(T_k(\hat{x}),T_k(\hat{y})) =  \left| \hat{x}-\hat{y} \right| \quad \forall  \hat{x},\hat{y} \in \overline{B}^e_{R_k} \qquad \Longleftrightarrow \qquad
\mathrm{dist}(x,y) =  \left| S_k(x)-S_k(y) \right| \quad \forall  x,y \in {\overline{A}_{t_k}}
\eeq
and
\beq \label{iso-2}
 T_k(S_k(x)) = x \quad \forall x \in   {\overline{A}_{t_k}} \qquad \Longleftrightarrow \qquad S_k(T_k(\hat{x})) = \hat{x} \quad  \forall \hat{x} \in \overline{B}^e_{R_k} \, .
\eeq
If we drop the request that the Euclidean balls are centered at a common given point, then up to a translation in $ \R^n $ (that may depend on $ k $) we can assume that for a fixed $ x_0 \in \overline{A}_{t_0} $ and a corresponding $ \hat{x}_0 \in \overline{B}^e_{R_0} $ it holds    
\beq \label{iso-3}  
T_k(\hat{x}_0)=x_0 \quad \Longleftrightarrow  \quad S_k(x_0)=\hat{x}_0 \qquad \forall k \in \N \, .
\eeq
For notational convenience, we will not change symbols and still refer to such translated maps and balls as $ T_k$, $S_k $ and $ \overline{B}^e_{R_k} $, respectively. Since $ \Mb^n $ is a complete manifold and $ \left\{ \overline{A}_{t_k} \right\} $ is an exhaustion of $ \Mb^n $ such that $ \overline{A}_{t_k} \Subset {A}_{t_{k+1}}  $, for all $ r>0 $ one can pick $ k_r \in \N $ (large enough) satisfying
\beq \label{iso-4}  
 B_r(x_0) \subset \overline{A}_{t_k}  \quad \Longleftrightarrow \quad B_r^e(\hat{x}_0) \subset \overline{B}^e_{R_k} \qquad \forall k \ge k_r \, .
\eeq
By virtue of \eqref{iso-1}, \eqref{iso-3} and \eqref{iso-4}, we are in position to apply Ascoli-Arzel\`a theorem to infer that there exist two maps $ T : \R^n \to \Mb^n $ and $ S: \Mb^n \to \R^n $ such that (up to a subsequence)
$$
\lim_{k \to \infty} T_k( \hat{x} ) = T(\hat{x}) \quad \forall \hat{x} \in \R^n \qquad \text{and} \qquad \lim_{k \to \infty} S_k( x ) = S( x ) \quad \forall {x} \in \Mb^n \, ,
$$
with both limits occurring locally uniformly. As a result, by passing to the limit in \eqref{iso-1} and \eqref{iso-2}, we infer that $ T $, along with its inverse $S$, is in fact a (metric) isometry between $ \R^n $ and $ \Mb^n $; thus it is also a smooth isometry between Riemannian manifolds (see \cite{MySt,Palais}), and the thesis follows. Note that \eqref{instant} is then a direct consequence of  the results of \cite{Aubin,Talenti}.
\end{proof} 

\begin{remark}
In the final part of the above proof we took advantage of a purely metric argument, which uses very little of the particular structure of a Cartan-Hadamard manifold (noncompactness and completeness). Nonetheless, it would have been possible to exploit a more geometric one, by observing that $ \{ A_{t_k} \} $ is an exhaustion of flat open sets of $ \Mb^n $, since each $ A_{t_k} $ is isometric to a Euclidean ball up to a negligible set. As a result, the simply connected manifold $ \Mb^n $ is flat and thus  isometric to $\R^n$ thanks to the well-known characterization of flat manifolds.
\end{remark}

\section{Radial solutions to the (critical or supercritical) $p$-Laplace equation}\label{sec: radial}

In this section we focus on \emph{radial positive solutions} to the $p$-Laplace equation \eqref{LE}:
\[
-\Delta_p u = u^q \, , \quad u>0 \, ,  \qquad \text{on $\Mb^n$} \, ,
\]
where $ \Mb^n $ is an $n$-dimensional Cartan-Hadamard \emph{model} manifold associated to a corresponding ``model function'' $ \psi $ as in  \eqref{metric} and $ q\ge p^\ast -1$. We will take these assumptions for granted from here on.

\subsection{Preliminaries and basic properties of radial solutions}\label{subs:1}

First of all, we consider the radial $p$-Laplace equation on $\Mb^n$ with positive initial datum, that is 
\beq\label{radial pb}
\begin{cases}
\left(\psi^{n-1} \left|u'\right|^{p-2} u' \right)' = -\psi^{n-1} \, |u|^{q-1} \, u \qquad \text{for $r>0$} \, , \\
u'(0) = 0 \, , \qquad u(0) =\alpha>0 \, .
\end{cases}
\eeq
Existence and uniqueness of a local classical solution $u \in C^1([0,T))$, with $w:=|u'|^{p-2} u' \in C^1((0,T))$, can be established as in the Euclidean case $\psi(r) = r$, studied in \cite{FLS96, GuVe88} (see in particular the appendices in those papers); this follows from the fact that $\psi$ is regular and $\psi(r) = r+o(r)$ as $r \to 0$. As in \cite[Lemma 1.1.1]{FLS96}, one can actually show that $w \in C^1([0,T))$, with $w'(0) = - \alpha^q/n<0$. Since $w(0) = 0$, we deduce that $w$, and hence $u'$, are strictly negative in a neighborhood of $r=0$. We aim to show that $u$ can be globally extended on the whole interval $[0,+\infty)$ remaining positive, with $u'<0$ on $(0,+\infty)$. To this end, following a similar strategy to \cite{BFG}, we will take advantage of a \emph{Pohozaev-type} technique. In the sequel, by ``maximal existence interval'', we mean the largest interval $ I \equiv [0,T) $ (with possibly $ T=+\infty $) where $u$ is a classical solution to \eqref{radial pb}. 

\begin{remark}\label{rem: reg rad}
Note that, if $u \in W^{1,p}_{\loc}(\Mb^n) \cap L^\infty_{\loc}(\Mb^n)$ is a radial weak solution to \eqref{LE}, then $u$ can be regarded as a solution to \eqref{radial pb} in the sense specified above. In the proofs of Theorems \ref{thm: rig}-\ref{thm: decay}, we will always implicitly use this fact. Indeed, by \cite{DiB, Tolks} we have that $u \in C^1([0,+\infty))$, and integrating the equation down to $r=0$ this ensures that $|u'|^{p-2} u' \in C^1([0,+\infty))$. In the special case $ q=p^\ast-1 $ local boundedness is actually for free (one can adapt the proof of Lemma \ref{lemma-bk} above), whereas it is well known that for $ q>p^\ast-1 $ locally unbounded radial solutions with local finite energy can exist (see for example~\cite[Section 9]{QS}).   
\end{remark}

Given a solution to \eqref{radial pb}, let us introduce the associated \emph{energy function}
$$
F_u(r) := \frac{p-1}p \left|u'(r)\right|^{p} + \frac1{q+1} \left|u(r)\right|^{q+1},
$$
along with the \emph{Pohozaev function}
$$
P_u(r) := \left( \int_0^r \psi^{n-1} \, ds \right) F_u(r) + \frac{\psi^{n-1}(r)}{q+1} \left|u'(r)\right|^{p-2} u(r) \, u'(r) \, .  
$$
Note that the differential equation in \eqref{radial pb} can equivalently be written as
\beq\label{rad eq}
\left(\left|u'\right|^{p-2} u'\right)' + (n-1) \, \frac{\psi'}{\psi} \left|u'\right|^{p-2} u' + \left|u\right|^{q-1} u = 0 \, .
\eeq

\begin{lemma}\label{lem: glob ex}
If $u$ is a solution to \eqref{radial pb} defined in its maximal existence interval $I$, then both $u$ and $u'$ remain bounded in $I$.
%Any solution to \eqref{radial pb} is globally defined. 
\end{lemma}
\begin{proof}
We have:
\[
\left|u'\right|^p = \left|\left|u'\right|^{p-2} u'\right|^\frac{p}{p-1} ,
\]
so that $ \left|u'\right|^p $ is also $ C^1(I) $ and
\[
\left(\left|u'\right|^p\right)' = \frac{p}{p-1} \left| \left|u'\right|^{p-2} u'\right|^\frac{2-p}{p-1} \left|u'\right|^{p-2} u' \left(\left|u'\right|^{p-2} u'\right)' = \frac{p}{p-1} \, u' \left( \left|u'\right|^{p-2} u'\right)' .
\]
Thanks to \eqref{rad eq} we thus obtain 
\beq\label{Fprimo}
F_u' = u' \left(\left|u'\right|^{p-2} u'\right)'  + \left|u\right|^{q-1} u u' = - (n-1) \frac{\psi'}{\psi} \left|u'\right|^p \le 0 \, ,
\eeq
whence it follows that $F_u(r) \le F_u(0) = \alpha^{q+1}/(q+1)$, for every $r \in I$. This clearly implies that both $u$ and $u'$ remain bounded in $I$. 
\end{proof}

Integrating the equation, it is readily seen that if $ u>0 $ in $ I' \subset I$ then $ u'<0 $ in $ I' \setminus \{ 0 \} $. Hence, due to Lemma \ref{lem: glob ex} and standard ODE theory, at this point we have only two alternatives: either $u$ exists in the whole interval $[0,+\infty)$ remaining positive, with $u'<0$ on $ (0,+\infty) $, or there exists  $R>0$ such that $u>0$ and $u'<0$ in $(0,R)$, $u(R) = 0$ and $u'(R) \le 0$. We will prove that only the first alternative is admissible; in order to establish it, we need the following lemma. 

\begin{lemma}\label{lem: on P}
If $u$ is a solution to \eqref{radial pb} defined in its maximal existence interval $I$, then 
\[
P_u'(r) = K(r) \left|u'(r)\right|^p \qquad \text{for every } r \in I \, ,
\]
where $ K $ is a suitable function depending only on $ q,p,\psi,n $, such that $K(r) \le 0$ for all $r \ge 0$ and $K(r) = 0$ for some $r>0$ if and only if $q=p^*-1$ and $\psi''(s) = 0$ for every $s \in (0,r)$. In particular, we have that $P_u(r) \le 0$ for every $r \in I$.
\end{lemma}

\begin{proof}
By direct computations, we have: 
$$
\begin{aligned}
P_u'(r) = & \, \psi^{n-1}(r) \, F_u(r) + \left( \int_0^r \psi^{n-1} \, ds \right) F_u'(r) + \frac{u(r)}{q+1} \left( \psi^{n-1} \left|u'\right|^{p-2} u'\right)'(r) + \frac{\psi^{n-1}(r)}{q+1} \left|u'(r)\right|^p \\
 = & \, \psi^{n-1}(r) \left(\frac{p-1}{p} + \frac{1}{q+1}\right) \left|u'(r)\right|^p + \frac{u(r)}{q+1} \left[ \psi^{n-1}(r) \left|u(r)\right|^{q-1} u(r) + \left( \psi^{n-1} \left|u'\right|^{p-2} u'\right)'(r) \right] \\
 & \, - \left( \int_0^r \psi^{n-1} \, ds\right) (n-1) \, \frac{\psi'(r)}{\psi(r)} \left|u'(r)\right|^p \\
= & \underbrace{\left[ \left(\frac{p-1}{p} + \frac{1}{q+1}\right) \psi^{n-1}(r) - \left( \int_0^r \psi^{n-1} \, ds \right) (n-1) \, \frac{\psi'(r)}{\psi(r)} \right]}_{=:K(r)} \left|u'(r)\right|^p ,
\end{aligned}
$$
where we have used the differential equation in \eqref{radial pb} and \eqref{Fprimo}. Integration by parts yields 
\[
\int_0^r \psi^{n-1} \, ds = \frac1{n} \int_0^r \frac{n \, \psi^{n-1} \psi'}{\psi'} \, ds = \frac{\psi^n(r)}{n \, \psi'(r)} + \frac 1 n \int_0^r \frac{\psi^n \, \psi''}{\left(\psi'\right)^2} \, ds \, .
\]
Therefore we can rewrite $K$ as
$$
K(r) =  \left(\frac{p-1}{p} + \frac{1}{q+1}- \frac{n-1}{n}\right) \psi^{n-1}(r) - \frac{n-1}{n} \, \frac{\psi'(r)}{\psi(r)} \int_0^r \frac{\psi^n \, \psi''}{\left(\psi'\right)^2} \, ds \, ,
$$
where the coefficient of $\psi^{n-1}(r)$ is smaller than or equal to $0$ since $q +1\ge p^*$ (with equality if and only if $q+1=p^*$), the second term is also nonpositive (recall that $\psi'(0) = 1$ and $\psi$ is convex) and can vanish at some $r>0$ if only if $\psi''(s) =0$ for every $s \in (0,r)$. Finally, the fact that $P_u \le 0$ in $I$ follows from $P_u(0) = 0$ along with the monotonicity of $P_u$. 
\end{proof}

A first relevant consequence of Lemma \ref{lem: on P} is that any solution to \eqref{radial pb} is global and remains positive in the whole $[0,+\infty)$.

\begin{lemma}\label{exist-glob}
If $u$ is a solution to \eqref{radial pb} defined in its maximal existence interval $I$, then $ I=[0,+\infty) $ with $ u>0 $ in $ [0,+\infty) $ and $ u'<0 $ in $(0,+\infty)$.
\end{lemma}
\begin{proof}
As observed above, either the claim is true or there exists $ R>0 $ such that $ u>0 $ and $ u'<0 $ in $ (0,R) $, $ u(R)=0 $ and $u'(R) \le 0$. The possibility that $u'(R) = 0$ can be immediately ruled out, by integrating the equation on $(0,R)$; while the fact that $u'(R) <0$ gives a contradiction with Lemma \ref{lem: on P}, since it would imply $ P_u(R)>0 $.
\end{proof}

We complete this subsection with two further useful properties of radial solutions. From here on we will take for granted that $ u $ is positive and globally defined in $ [0,+\infty) $. 

\begin{lemma}\label{lem: monot and lim}
If $u$ is a solution to \eqref{radial pb}, then $u(r) \to \lambda \in [0,u(0))$ and $u'(r) \to 0$ as $r \to +\infty$.
\end{lemma}

\begin{proof}
The monotonicity and positivity of $u$ ensure that $u(r) \to \lambda \in [0,u(0))$ as $r \to +\infty$, whence
\beq\label{liminf 0}
\liminf_{r \to +\infty} \left|u'(r)\right| = 0 \,  .
\eeq
We are left with proving that actually $u'(r) \to 0$ as $r \to +\infty$. To this end, we integrate both sides in \eqref{Fprimo} on an interval $(r_0,r)$, with $0<r_0<r$, deducing that
\[
F_u(r) = F_u(r_0) - (n-1) \int_{r_0}^r \frac{\psi'}{\psi} \left|u'\right|^p ds \quad \Rightarrow \quad \frac{p-1}{p} \left|u'(r)\right|^p = F_u(r_0) - \frac{u^{q+1}(r)}{q+1}  - (n-1) \int_{r_0}^r \frac{\psi'}{\psi} \left|u'\right|^p ds .
\]
Clearly the right-hand side in the last identity has a limit as $r \to +\infty$, and so does the left-hand side. This means that $|u'(r)|$ itself has a limit as $r \to +\infty$, which thanks to \eqref{liminf 0} completes the proof.
\end{proof}

\begin{lemma}\label{lem: P strict neg}
Let $\Mb^n \not \equiv \R^n$ and $u$ be a solution to \eqref{radial pb}. Then there exists $\overline{r}>0$ such that
\[
P_u(r) <0 \qquad \forall r > \overline{r} \, .
\]
\end{lemma}
\begin{proof}
In terms of $\psi$, the assumption $\Mb^n \not \equiv \R^n$ is equivalent to $\psi''(r)>0$, at least for every $r$ ranging in an open interval $(r_1,r_2)$. Therefore, by Lemmas \ref{lem: on P} and \ref{exist-glob} we have that $P_u'(r) = K(r) \left|u'(r)\right|^p < 0$ for all $r>r_2$, whence the thesis follows.
\end{proof}

It is convenient to sum up what we have proved so far. We have shown that, on any Cartan-Hadamard model manifold, there exist (unique) solutions to the radial problems \eqref{radial pb} which are globally defined, remain positive and decrease, with $u(r) \to \ell \in [0,u(0))$ and $u'(r) \to 0$ as $r \to +\infty$. Moreover, the Pohozaev function $P_u$ is nonincreasing, nonpositive and, if $\Mb^n \not \equiv \R^n$, strictly negative for large $r$. To proceed further, we now distinguish between the $p$-stochastically complete and incomplete cases.

\subsection{Proof of the main results for \texorpdfstring{$\boldsymbol p$}{}-stochasically \emph{complete} manifolds.}

Throughout this whole subsection we assume that the Cartan-Hadamard manifold at hand is $p$-stochastically complete, namely \eqref{p-sc} holds. 

\begin{proof}[Proof of Theorem \ref{thm: sc vs si}-($i$)]
By what we have established in Subsection \ref{subs:1}, we can assert that for all $\alpha>0 $ there exists a unique solution to \eqref{radial pb}, which is in fact classical and complies with \eqref{LE}; vice versa, any radial solution to \eqref{LE} satisfies \eqref{radial pb} for some $ \alpha>0 $. Because different values of $ \alpha $ give rise to different solutions, there are infinitely-many such solutions.

Let now $u$ be any solution to \eqref{radial pb}. Since $P_u(r) \le 0$ for all $r>0$, in particular we have that
\[
\left(\int_0^r \psi^{n-1} \, ds \right) u^{q+1}(r) - \psi^{n-1}(r) \left[-u'(r)\right]^{p-1} u(r) \le 0 \qquad \forall r > 0
\]
(recall that $u>0$ and $u' < 0$). From the above inequality, straightforward computations give
\[
\frac{\left[-u'(r)\right]^{p-1}}{u^q(r)} \ge \psi^{1-n}(r) \int_0^r \psi^{n-1} \, ds \quad \implies \quad -\frac{u'(r)}{u^{\frac{q}{p-1}}(r)} \ge \Theta^{\frac{1}{p-1}}(r) \qquad \forall r > 0 \, .
\]
By integrating both sides on $(0,r)$, we deduce that
\[
\frac{p-1}{q+1-p} \left( \frac{1}{u^{\frac{q+1-p}{p-1}}(r)} - \frac{1}{\alpha^{\frac{q+1-p}{p-1}}  } \right) \ge \int_0^r \Theta^\frac{1}{p-1}\,ds \to +\infty \qquad \text{as } r \to +\infty \, ,
\]
thanks to assumption \eqref{p-sc}. Therefore $u(r) \to 0$ as $r \to +\infty$, and more precisely
\beq\label{24021}
u(r) \le \left( \frac{p-1}{q+p-1} \right)^{\frac{p-1}{q+p-1}} \left(\int_0^r \Theta^\frac{1}{p-1} \, ds \right)^{-\frac{p-1}{q+1-p}} \qquad \forall r>0 \, .  \qedhere
\eeq
\end{proof}

\begin{remark}
In order to prove Theorem \ref{thm: sc vs si}-($i$), we could have also argued in the following way: since we know that any radial (positive) global solution is decreasing, it is enough to show that $\inf_{\Mb^n} u = 0$. To this end, the $p$-stochastic completeness of $\Mb^n$ allows us to apply the weak maximum principle at infinity in \cite[Theorem 1.2]{MarVal} to $-u$, whence the fact that $u(r) \to 0$ as $r \to +\infty$ follows. However, in the sequel we will need the decay estimate \eqref{24021}.
\end{remark}

Having established that solutions vanish at infinity, we are ready to prove Theorem \ref{thm: rig} in the $p$-stochastically complete case. 

\begin{proof}[Proof of Theorem \ref{thm: rig} under \eqref{p-sc}]
Suppose by contradiction that there exists a radial solution $u$ to \eqref{LE}, satisfying \eqref{hp grad}, on a Cartan-Hadamard model manifold $\Mb^n \not \equiv \R^n$ complying with \eqref{p-sc}. In particular, we know that $ u $ is a (classical) solution to \eqref{radial pb} for some $\alpha>0$. Hence, by virtue of Lemmas \ref{lem: glob ex}--\ref{lem: P strict neg} and Theorem \ref{thm: sc vs si}-($i$), we can assert that $u(r) \to 0$ as $r \to +\infty$ and $P_u(r) \le P_u(\overline r) =: - C<0$ for all $r> \overline r$, for some $\overline r>0$. By the definition of $P_u$, this yields
\[
\frac{\psi^{n-1}(r)}{q+1} \left|u'(r)\right|^{p-2} u'(r) \, u(r) \le - C \qquad  \forall r > \overline{r} \, .
\]
That is, since $u'<0$ and $u>0$, 
\[
\psi^{n-1}(r) \left|u'(r)\right|^p \ge - C \, \frac{u'(r)}{u(r)} \qquad \forall r > \overline{r} \, .
\]
Upon integrating on $(\bar r, r)$, we deduce that
\[
\int_{\overline r}^r \left|u'\right|^p \psi^{n-1} \, ds \ge C \log ( u(\overline{r})) - C \log (u(r)) \to +\infty \qquad \text{as } r \to +\infty \, ;
\]
on the other hand $\int_{\overline r}^{+\infty} \left|u'\right|^p \psi^{n-1} \, ds $ must be finite by \eqref{hp grad}, which leads to the desired contradiction.

\end{proof}

We now address the proof of Theorem \ref{thm: decay}-$(i)$. To this end, recall that we require the additional assumptions \eqref{hp add 1} and \eqref{hp add 2}. We will prove the result through a series of lemmas, following a similar strategy to the one developed in \cite{BFG}, where the case $p=2$ is treated. Let us start with some preliminary observations.

\begin{remark}\label{rem: hp add}
Assumption \eqref{hp add 1} implies in particular that there exist two constants $c_1, c_2 >0$ such that, for all $\delta>0$, one can pick $r_\delta>0$ so large that
$$
c_{1} (1-\delta) \, e^{c_2(1-\delta) \, r^{1-\gamma} }\le \psi(r) \le c_{1}(1+\delta) \, e^{c_2(1+\delta) \, r^{1-\gamma}} \qquad \forall r>r_\delta \, . 
$$
On the other hand, assumption \eqref{hp add 2} ensures that for all $M>0$ there exist $ C_M,r_M>0 $ such that 
\[
\psi(r) \ge C_M \, e^{M r} \qquad \forall r > r_M \, . 
\]
In both cases $\psi$ has at least an exponential-like growth at infinity. We also notice that, under either \eqref{hp add 1} or \eqref{hp add 2}, it holds
\beq \label{asymp1}
\frac{\psi^{n-1}(r)}{\int_0^r \psi^{n-1} \, ds } \sim (n-1) \, \frac{\psi'(r)}{\psi(r)} \qquad \text{as } r \to +\infty \, ,
\eeq
where by the symbol $ \sim $ we mean that the ratio tends to $1$. Indeed, in case \eqref{hp add 1} is satisfied, then by L'H\^opital's rule we have
\beq \label{asymp-theta}
\frac{r^\gamma \, \psi^{n-1}(r)}{\int_0^r \psi^{n-1} \, ds} \sim (n-1) \, \frac{r^\gamma \, \psi'(r)}{\psi(r)} + \gamma \, r^{\gamma-1} \sim (n-1) \, \frac{r^\gamma \, \psi'(r)}{\psi(r)} \to (n-1) \, \ell \qquad \text{as } r \to +\infty \, ;
\eeq
whereas in case \eqref{hp add 2} is satisfied, still L'H\^opital's rule yields 
$$
\frac{\frac{\psi(r)}{\psi'(r)} \, \psi^{n-1}(r)}{\int_0^r \psi^{n-1} \, ds} \sim   \frac{\left( \frac{\psi(r)}{\psi'(r)} \right)' \psi^{n-1}(r) + (n-1)\,\psi^{n-1}(r)}{\psi^{n-1}(r)} \to n-1 \qquad \text{as } r \to +\infty \, .
$$
\end{remark}

\begin{lemma}\label{lem: no very fast decay}
Let $u$ be a radial solution to \eqref{LE}. Suppose that \eqref{p-sc} and either \eqref{hp add 1} or \eqref{hp add 2} hold. Then there exist no positive constants $C, \beta>0$ such that $u(r) \le C \, \psi^{-\beta}(r)$ for all $r>0$.
\end{lemma} 
\begin{proof}
Assume by contradiction that there exist two constants $C, \beta>0$ as in the statement. For the sake of readability, along the proof $C$ will stand for a general constant, which may actually change from line to line but will not be relabeled. Since $ u $ is bounded and $ \psi(r) \to +\infty $ as $ r \to +\infty $, it is not restrictive to assume further that $\beta<(n-1)/q$. By integrating \eqref{radial pb} on $(0,r)$, we thus find that
\begin{equation*}\label{23021}
-u'(r) =  \left( \psi^{1-n}(r) \int_0^r \psi^{n-1} \, u^q \, ds \right)^\frac{1}{p-1} \le C \left( \psi^{1-n}(r) \int_0^r \psi^{n-1-\beta q} \, ds \right)^\frac{1}{p-1} \qquad \forall r>0 \, ;
\end{equation*}
a subsequent integration on $(r, +\infty)$ yields
\beq\label{20110}
u(r) \le C \, \int_r^{+\infty} \left( \psi^{1-n}(s) \int_0^s \psi^{n-1-\beta q} \, dt \right)^\frac{1}{p-1} ds \qquad \forall r>0 \, ,
\end{equation}
where we used the fact that $u(r) \to 0 $ as $r \to +\infty$ (Theorem \ref{thm: sc vs si}-($i$)). Note that the integral on the right-hand side of \eqref{20110} vanishes as $r \to +\infty$: this can be seen, for instance, upon bounding the innermost integral by $ s \, \psi^{n-1-\beta q}(s) $ and exploiting the exponential-like growth of $ \psi $, as observed in Remark \ref{rem: hp add}. 

We claim that, for every $ \eps \in (0,q+1-p) $, there exists $ C_\eps>0 $ (which neither will be relabeled from line to line) such that
\beq\label{20111}
\frac{1}{\psi^{-\frac{\beta q}{p-1+\eps}}(r)} \int_r^{+\infty} \left( \psi^{1-n}(s) \int_0^s \psi^{n-1-\beta q} \, dt \right)^\frac{1}{p-1} ds \le C_\eps  \qquad \text{for $r$ large enough} \, .
\eeq
In order to prove \eqref{20111}, it suffices to apply L'H\^opital's rule to the ratio on the left-hand side and take advantage of \eqref{asymp1} (which actually holds for any real $n>1$) along with the just recalled exponential-like growth of $ \psi $, to deduce that in fact such ratio vanishes at infinity. In view of \eqref{20110}, estimate \eqref{20111} then entails $u(r) \le C_\eps \, \psi^{-{\beta q}/{(p-1+\eps)}}(r)$, which is stronger than the initial bound since $ q /(p-1) > 1$. We can therefore iterate the previous argument a finite number of times, inferring that
$$
u(r) \le C_\eps \, \psi^{-\beta \left(\frac{q}{p-1+\eps}\right)^k}(r) \quad \forall r>0 \, , \quad \text{for every $k$ satisfying} \quad  \beta \left(\frac{q}{p-1+\eps}\right)^{k-1}<\frac{n-1}{q} \, .
$$ 
In particular, because $ \beta $ can be taken as small as needed,
we can assert that
\beq\label{20112}
u(r) \le C_\eps \, \psi^{- \frac{n-1-\eps}{p-1+\eps}}(r) \qquad \forall r>0 \, .
\eeq
We will now reach a contradiction by obtaining an incompatible estimate in the opposite direction. To this end, recall that $P_u(r) \le 0 $ in view of Lemma \ref{lem: on P}. As a result,
\[
\left( \int_0^r \psi^{n-1} \, ds \right) \frac{p-1}{p} \left|u'(r)\right|^p  + \frac{\psi^{n-1}(r)}{q+1} \left|u'(r)\right|^{p-2} u(r) \, u'(r) \le  0 \qquad \forall r > 0 \, .
\]
Since $u'<0$ and $u>0$, we deduce that there exists $r_\eps>0$ such that
\beq\label{20113}
\frac{u'(r)}{u(r)} \ge - \frac{p}{(p-1)(q+1)} \, \frac{\psi^{n-1}(r)}{\int_0^r \psi^{n-1} \, ds} >  - \frac{p (n-1+\eps)}{(p-1)(q+1)}  \, \frac{\psi'(r)}{\psi(r)} \qquad \forall r > r_\varepsilon \, ,
\eeq
the last inequality following from \eqref{asymp1}. By integrating \eqref{20113} on $(r_\eps, r)$, we conclude that 
$$
u(r) \ge C_\eps' \, \psi^{-\frac{p(n-1+\eps)}{(p-1)(q+1)}}(r) \qquad \forall r > r_\varepsilon \, ,
$$
where $C_\eps'$ is another suitable positive constant. A comparison with \eqref{20112} gives the desired contradiction, since $\eps>0$ can be chosen so small that $p(n-1+\eps)(p-1+\eps) < (n-1-\eps)(p-1)(q+1)$.
\end{proof}

\begin{lemma}\label{lem: decay 2}
Let $u$ be a radial solution to \eqref{LE}. Suppose that \eqref{p-sc} and either \eqref{hp add 1} or \eqref{hp add 2} hold. Then
\beq	\label{fund-1}
\lim_{r \to +\infty} \frac{u'(r)}{u(r)} \frac{\psi(r)}{\psi'(r)} = 0 \, .
\eeq
\end{lemma} 
\begin{proof}
We consider at first the case when \eqref{hp add 2} holds, and start by showing that
\[
\limsup_{r \to +\infty} \frac{u'(r)}{u(r)} \frac{\psi(r)}{\psi'(r)} = 0 \, .
\]
Indeed, if this $\limsup$ were smaller than $-\beta<0$, then it would be immediate to deduce that $u(r) \le C \, \psi^{-\beta}(r)$ for a suitable $C>0$, in contradiction with Lemma \ref{lem: no very fast decay}. Therefore, it remains to establish that the $ \limsup $ is in fact a limit. To this aim, we proceed again by contradiction. Should the limit not exist, there would be a sequence $r_m \to +\infty$ of local minimizers for the differentiable function $\frac{u'}{u} \frac{\psi}{\psi'}$, along which the latter does not tend to $0$. Let us set
\beq\label{20114}
\Gamma(r):= \left[ \log \left( \frac{\psi'(r)}{\psi(r)}\right)\right]'.
\eeq
It is easy to check, by extremality, that for all $ m \in \N $
\[
u''(r_m) \, u(r_m) = \left(u'(r_m)\right)^2 + u(r_m) \, u'(r_m) \, \Gamma(r_m)
\]
(note that, since $u'<0$ on $(0,+\infty)$, $u$ is actually $C^2((0,+\infty))$). If we multiply \eqref{rad eq} by $u$, this identity reads
\[
\left(-u'(r_m)\right)^{p-1} = \frac{u^{q+1}(r_m)}{(p-1) \, u'(r_m) + (p-1) \, u(r_m) \, \Gamma(r_m) + (n-1) \, \frac{\psi'(r_m)}{\psi(r_m)} \, u(r_m)} \, .
\]
Recall that for every $\eps>0$ small enough estimate \eqref{20113} holds at $ r = r_m $ for all $ m $ sufficiently large, whence
\[
\begin{aligned}
\left(-u'(r_m)\right)^{p-1} \le & \, \frac{u^q(r_m)}{-\frac{p}{q+1} \, (n-1+\eps) \, \frac{\psi'(r_m)}{\psi(r_m)}  + (p-1) \, \Gamma(r_m) + (n-1) \frac{\psi'(r_m)}{\psi(r_m)}} \\
 \le & \, \frac{u^q(r_m)}{C \, \frac{\psi'(r_m)}{\psi(r_m)} \left( 1+o(1) \right) } 
\end{aligned}
\]
as $ m \to \infty$, where we have used assumption \eqref{hp add 2} and the fact that $p<q+1$. Here and in the sequel $ C $ stands for a generic positive constant, whose explicit value is irrelevant to our purpose. As a consequence, 
\[
-\frac{u'(r_m)}{u(r_m)}\frac{\psi(r_m)}{\psi'(r_m)} \le C \, u^\frac{q+1-p}{p-1}(r_m) \left(\frac{\psi(r_m)}{\psi'(r_m)}\right)^{1+ \frac{1}{p-1}}
\]
for all $m$ large enough. Since $u(r) \to 0$ as $ r \to +\infty $, and \eqref{hp add 2} holds, this implies that the limit of the left-hand side is also $0$, in contradiction with the definition of the sequence $\{r_m\}$.

Let us turn to the case when \eqref{hp add 1} holds, and proceed similarly. Note that \eqref{fund-1} amounts to
$$
\lim_{r \to +\infty} \frac{r^\gamma  \,u'(r)}{u(r)}  = 0 \, .
$$
The fact that the $ \limsup $ is zero can be shown exactly as in the previous case. In order to establish that the above limit does exist, we argue again by contradiction, assuming that there is a sequence $r_m \to +\infty$ of local minimizers for the differentiable function $\frac{r^\gamma \, u'(r)}{u(r)}$, along which the latter does not tend to $0$. By extremality, for all $ m \in \N $ we have 
\[
u''(r_m) \, u(r_m) = \left(u'(r_m)\right)^2 - \frac{\gamma}{r_m} \, u(r_m) \, u'(r_m) \, .
\]
Hence, upon multiplying \eqref{rad eq} by $u$, this identity gives
\[
\left(-u'(r_m)\right)^{p-1} = \frac{u^{q+1}(r_m)}{(p-1) \, u'(r_m) - (p-1) \, \frac{\gamma}{r_m} \, u(r_m)  + (n-1) \, \frac{\psi'(r_m)}{\psi(r_m)} \, u(r_m) } \, .
\]
Thanks to \eqref{20113} and \eqref{hp add 1} (recall that $ \gamma < 1 $), we thus deduce that for every $ \eps>0 $ small enough
\[
\left(-u'(r_m)\right)^{p-1} 
\le  \frac{u^q(r_m)}{-\frac{p}{q+1} \, (n-1+\eps) \, \frac{\psi'(r_m)}{\psi(r_m)}  + (n-1) \, \frac{\psi'(r_m)}{\psi(r_m)} + o\!\left( \frac{\psi'(r_m)}{\psi(r_m)} \right)} \le C \, r_m^\gamma \, u^q(r_m)
\]
for all $m$ sufficiently large. Hence,
\beq\label{24022}
-\frac{r_m^\gamma \, u'(r_m)}{u(r_m)} \le C \, u^{\frac{q+1-p}{p-1}}(r_m) \, r_m^{\frac{p \gamma}{p-1}} \le C \left( \int_0^{r_m} \Theta^{\frac{1}{p-1}} \, ds \right)^{-1} r_m^{\frac{p \gamma}{p-1}} \,,
\eeq
where in the last step we used estimate \eqref{24021}. Due to \eqref{asymp-theta}, we infer that
\[
\int_0^r \Theta^{\frac{1}{p-1}} \, ds \sim C \int_0^r s^{\frac{\gamma}{p-1}}\,ds \sim C \, r^{\frac{\gamma}{p-1}+1} \qquad \text{as } r \to +\infty \, ,
\]
which plugged in \eqref{24022} yields
\[
-\frac{r_m^\gamma \, u'(r_m)}{u(r_m)} \le C \, r_m^{\gamma-1} \to 0 \qquad \text{as } m \to \infty \, ,
\]
in contradiction with the definition of the sequence $\{r_m\}$. 
\end{proof}

\begin{lemma}\label{lem: decay 3}
Let $u$ be a radial solution to \eqref{LE}. Suppose that \eqref{p-sc} and either \eqref{hp add 1} or \eqref{hp add 2} hold. Then
\beq\label{21110}
\lim_{r \to +\infty} \frac{\left(-u'(r)\right)}{u^\frac{q}{p-1}(r)} \left(\frac{\psi'(r)}{\psi(r)}\right)^\frac{1}{p-1} = \left(\frac{1}{n-1}\right)^\frac{1}{p-1}.
\end{equation}
\end{lemma}
\begin{proof}
We assume at first that \eqref{hp add 2} holds. To begin with, let us prove by contradiction the existence of the limit in \eqref{21110}. Should the latter not exist, then there would be a sequence $r_m \to +\infty$ of local maxima/minima points for the $C^1$ function 
$$
\Phi(r) := \frac{\left(-u'(r)\right)}{u^\frac{q}{p-1}(r)} \left(\frac{\psi'(r)}{\psi(r)}\right)^\frac{1}{p-1} 
$$
such that $\{ \Phi(r_m) \}$ does not have a limit.
By extremality, for all $ m \in \N $ we have
\[
u''(r_m) \, u(r_m) = \frac{q}{p-1} \left(u'(r_m)\right)^2 - \frac{u(r_m) \, u'(r_m)}{p-1} \, \Gamma(r_m) \, ,
\]
where $\Gamma$ is defined in \eqref{20114}. Up to multiplying \eqref{rad eq} by $u$, this identity entails
\[
\left(-u'(r_m)\right)^{p-1} = \frac{u^{q}(r_m)}{q \, \frac{u'(r_m)}{u(r_m)} - \Gamma(r_m) + (n-1) \, \frac{\psi'(r_m)}{\psi(r_m)}} \, ,
\]
whence
\[
\Phi^{p-1}(r_m) = \frac{\left(-u'(r_m)\right)^{p-1}}{u^q(r_m)} \, \frac{\psi'(r_m)}{\psi(r_m)} = \left[q \, \frac{u'(r_m)}{u(r_m)} \, \frac{\psi(r_m)}{\psi'(r_m)} - \frac{\psi(r_m)}{\psi'(r_m)} \, \Gamma(r_m) + n-1 \right]^{-1} \to \frac{1}{n-1}
\]
as $ m \to \infty $, where in the last passage we have used Lemma \ref{lem: decay 2} and assumption \eqref{hp add 2}. However, by assumption $\{\Phi(r_m)\}$ does not have a limit, thus we have reached a contradiction and the existence of the limit in \eqref{21110} is proved. Now let us compute it. By \eqref{rad eq} we have
\beq\label{24023}
(p-1)\,\frac{u''(r)}{\left(-u'(r)\right)} \frac{\psi(r)}{\psi'(r)} - (n-1) + \frac{u^q(r)}{\left(-u'(r)\right)^{p-1}}\frac{\psi(r)}{\psi'(r)} = 0 \qquad \forall r >0 \, ;
\eeq
what we have proved so far implies that the limit as $r \to +\infty$ of the first term on the left-hand side does exist as well. On the other hand, assumption \eqref{hp add 2} ensures that
\[
\lim_{r \to +\infty} \frac{u''(r)}{u'(r)} \frac{\psi(r)}{\psi'(r)} = \lim_{r \to +\infty} \left( \frac{u''(r)}{u'(r)} - \Gamma(r) \right) \frac{\psi(r)}{\psi'(r)}  =  \lim_{r \to +\infty} \frac{\left[u'(r) \, \psi(r) \left(\psi'(r)\right)^{-1} \right]'}{u'(r)} \, .
\]
This implies that the rightmost limit does exist and hence, by Lemma \ref{lem: decay 2} and L'H\^opital's rule,
\[
 0 = \lim_{r \to +\infty}\frac{u'(r)}{u(r)} \frac{\psi(r)}{\psi'(r)} =\lim_{r \to +\infty} \frac{\left[u'(r) \, \psi(r) \left(\psi'(r)\right)^{-1} \right]'}{u'(r)} = \lim_{r \to +\infty} \frac{u''(r)}{u'(r)} \frac{\psi(r)}{\psi'(r)} \, .
\] 
Taking advantage of this information in \eqref{24023}, we necessarily obtain
\[
\lim_{r \to +\infty} \frac{u^q(r)}{\left(-u'(r)\right)^{p-1}}\frac{\psi(r)}{\psi'(r)} = n-1 \, ,
\]
which is equivalent to \eqref{21110}. 

The case when \eqref{hp add 1} holds can be treated in a similar way. To prove the existence of the limit in \eqref{21110}, we argue again by contradiction: claiming that the latter does not exist is equivalent to admitting that one can pick a sequence $r_m \to +\infty$ of local maxima/minima points for the $C^1$ function
$$
\Psi(r) := \frac{\left(-u'(r)\right)}{u^\frac{q}{p-1}(r)} \, r^{-\frac{\gamma}{p-1}}
$$
such that $\{\Psi(r_m)\}$ does not have a limit. By extremality, for all $ m \in \N $ we have
\[
u''(r_m) \, u(r_m) = \frac{q}{p-1} \left(u'(r_m)\right)^2 + \frac{\gamma \, u(r_m) \, u'(r_m)}{(p-1)\, r_m} \, .
\]
From this identity, using \eqref{rad eq} multiplied by $u$, we easily deduce that
\[
\left(-u'(r_m)\right)^{p-1} = \frac{u^q(r_m)}{q \, \frac{u'(r_m)}{u(r_m)} + \frac{\gamma}{r_m} + (n-1) \, \frac{\psi'(r_m)}{\psi(r_m)}  } \, ,
\]
which (thanks to Lemma \ref{lem: decay 2}) in turn yields $\Psi^{p-1}(r_m) \to 1/[(n-1)\ell]$ as $ m \to \infty $, a contradiction. Thus, the existence of the limit of $ \Psi(r) $ as $ r \to +\infty $ is proved. In order to compute it, let us observe that still by \eqref{rad eq} we have
\beq\label{24026}
(p-1) \, \frac{ r^\gamma \, u''(r)}{\left(-u'(r)\right)} - (n-1)  \, \frac{r^\gamma \, \psi'(r)}{\psi(r)} + \frac{r^\gamma \, u^q(r)}{\left(-u'(r)\right)^{p-1}} = 0 \qquad \forall r >0 \, .
\eeq
What we have proved so far, along with \eqref{hp add 1}, ensures the existence of the limit of the first term on the left-hand side, which can be rewritten as
\[
\lim_{r \to +\infty} \frac{r^\gamma \, u''(r)}{u'(r)}  = \lim_{r \to +\infty} r^\gamma \left( \frac{u''(r)}{u'(r)} + \frac{\gamma}{r} \right)  =  \lim_{r \to +\infty} \frac{\left(r^\gamma \, u'(r)\right)'}{u'(r)} \, .
\]
Therefore, by Lemma \ref{lem: decay 2} and L'H\^opital's rule, we end up with
\[
 0 = \lim_{r \to +\infty}\frac{r^\gamma \, u'(r)}{u(r)}  =\lim_{r \to +\infty} \frac{\left(r^\gamma \, u'(r)\right)'}{u'(r)} = \lim_{r \to +\infty} \frac{r^\gamma \, u''(r)}{u'(r)} \, ,
\] 
so that \eqref{24026} yields  
\[
\lim_{r \to +\infty} \frac{r^\gamma \, u^q(r)}{\left(-u'(r)\right)^{p-1}} = (n-1) \ell \, ,
\]
which is equivalent to the desired result when assumption \eqref{hp add 1} holds.
\end{proof}

We are now in position to prove the asymptotics of solutions under either \eqref{hp add 1} or \eqref{hp add 2}.
 
\begin{proof}[Proof of Theorem \ref{thm: decay}-($i$)]
By Lemma \ref{lem: decay 3}, for every $\eps>0$ small enough there exists $r_\eps>0$ such that
\[
\left[\left(\frac{1}{n-1}-\eps\right)\frac{\psi(r)}{\psi'(r)}\right]^\frac{1}{p-1} < - \frac{u'(r)}{u^\frac{q}{p-1}(r)} < \left[\left(\frac{1}{n-1}+\eps\right)\frac{\psi(r)}{\psi'(r)}\right]^\frac{1}{p-1} \qquad \forall r > r_\eps \, ,
\]
whereas, for a possibly larger $ r_\eps $, we have
\[
(n-1-\eps) \, \Theta(r) <\frac{\psi(r)}{\psi'(r)} < (n-1+\eps) \, \Theta(r) \qquad \forall r > r_\eps
\]
in view of \eqref{asymp1}. Therefore, by combining the above estimates, we deduce that for a suitable $C>0$ (depending only on $n$) it holds 
\[
(1-C \eps) (\Theta(r))^\frac{1}{p-1} < - \frac{u'(r)}{u^\frac{q}{p-1}(r)} < (1+C \eps) (\Theta(r))^\frac{1}{p-1} \qquad \forall r>r_\eps \, .
\]
By integrating, we obtain
\[
(1-C \eps) \, \frac{q+1-p}{p-1} \int_{r_\eps}^r \Theta^\frac{1}{p-1} \, ds < u^{-\frac{q+1-p}{p-1}}(r) - u^{-\frac{q+1-p}{p-1}}(r_\eps) < (1+C \eps) \, \frac{q+1-p}{p-1} \int_{r_\eps}^r \Theta^\frac{1}{p-1} \, ds  \, ,
\]
for every $r>r_\eps$, that is
\begin{multline*}
 \left[ 1+C \eps + \frac{p-1}{q+1-p} \left( \int_{r_\eps}^r \Theta^\frac{1}{p-1} \, ds \right)^{-1} u^{-\frac{q+1-p}{p-1}} (r_\eps) \right]^{-1} 
 < \,  \frac{q+1-p}{p-1} \left( \int_{r_\eps}^{r} \Theta^{\frac{1}{p-1}} \, ds \right) u^{\frac{q+1-p}{p-1}}(r) \\
 < \, \left[ 1-C \eps + \frac{p-1}{q+1-p} \left( \int_{r_\eps}^r \Theta^\frac{1}{p-1} \, ds \right)^{-1} u^{-\frac{q+1-p}{p-1}} (r_\eps) \right]^{-1} 
\end{multline*}
for all $ r>r_\eps $. The thesis follows by letting first $ r \to +\infty $ and then $ \eps \to 0 $, using assumption \eqref{p-sc}.
\end{proof}

Finally, we show that without assuming \eqref{hp add 1} or \eqref{hp add 2} the just established asymptotic behavior fails.

\begin{proof}[Proof of Proposition \ref{prop: oscillazioni}]
As a key starting point, we claim that thanks to \eqref{hp fail} there exist $ \hat{\kappa} > 0 $ and $  \hat{r}>0 $ such that 
\beq\label{ee1}
\int_0^r \Theta^{\frac{1}{p-1}}(s) \left( \int_0^s \Theta^{\frac{1}{p-1}} \, dt \right)^{-\frac{(p-1)q}{q+1-p}-1} \int_0^s \psi^{n-1} \, dt \, ds \ge \hat{\kappa} \, \int_0^r \psi^{n-1} \, ds \left( \int_0^r \Theta^{\frac{1}{p-1}} \, ds \right)^{-\frac{(p-1)q}{q+1-p}}
\eeq
for every $r \ge \hat{r}$. Indeed, formula \eqref{hp fail} can be rewritten as
\beq \label{useful}
\liminf_{r \to +\infty}  \frac{\int_0^r \Theta^{\frac{1}{p-1}}(s) \int_0^s \psi^{n-1} \, dt \, ds }{\int_0^r \Theta^{\frac{1}{p-1}} \, ds \, \int_0^r \psi^{n-1} \, ds } > 0 \quad \iff \quad \frac{\int_0^r \Theta^{\frac{1}{p-1}}(s) \int_0^s \psi^{n-1} \, dt \, ds}{\int_0^r \Theta^{\frac{1}{p-1}} \, ds \, \int_0^r \psi^{n-1} \, ds} \ge  \hat{\kappa}
\eeq
for every $r \ge \hat{r}$,
with constants $ \hat{\kappa} , \hat{r} > 0 $ as above, from which \eqref{ee1} easily follows by monotonicity of the innermost integral involving $ \Theta $. 

The validity of \eqref{p-sc} under \eqref{hp fail} is a consequence of \eqref{useful}, since the latter is equivalent to
$$
\left[\log\left( \int_0^r \Theta^{\frac{1}{p-1}} \, ds  \right) \right]' \ge \hat{\kappa} \left[\log\left( \int_0^r \Theta^{\frac{1}{p-1}}(s) \int_0^s \psi^{n-1} \, dt \, ds \right) \right]'  \qquad \forall r \ge \hat{r} 
$$
and the monotonicity of $ \psi $ ensures that 
$$
\int_0^{r} \Theta^{\frac{1}{p-1}}(s) \int_0^s \psi^{n-1} \, dt \, ds \ge \frac{p-1}{2p-1} \, \frac{\left(\int_0^r \psi^{n-1} \, ds \right)^{\frac{2p-1}{p-1}}}{\psi^{\frac{p(n-1)}{p-1}}(r)} \qquad \forall r>0 \, .
$$
If the ratio on the right-hand side were bounded, this would imply that $ \int_0^{r_0} \psi^{n-1} \, ds = +\infty $ at some finite $ r_0>0 $, which is absurd.

In order to establish that \eqref{hp add 3} cannot hold, we can argue by contradiction. Should such an asymptotic behavior be true, then upon integrating \eqref{radial pb} at infinity we would end up with the identity 
\beq\label{ee2}
\lim_{r \to +\infty} \left( \int_0^r  \Theta^{\frac{1}{p-1}} \, ds \right)^\frac{p-1}{q+1-p} \int_r^{+\infty} \left( \frac{\int_0^s \, \psi^{n-1} \, u^q \, dt }{\psi^{n-1}(s)} \right)^{\frac{1}{p-1}} ds =  \left(\frac{p-1}{q+1-p}\right)^\frac{p-1}{q+1-p} .
\eeq
Now we observe that 
\beq\label{ee2-bis}
\lim_{r \to +\infty} \left( \int_0^r  \Theta^{\frac{1}{p-1}} \, ds \right)^\frac{p-1}{q+1-p} \, \int_r^{+\infty} \frac{1}{\psi^{\frac{n-1}{p-1}}} \, ds = 0 \, .
\eeq
This is a direct consequence of the fact that $ \psi'(r) \ge 1 $ and $ \psi(r) \ge r $:
$$
\begin{aligned}
\left( \int_0^r  \Theta^{\frac{1}{p-1}} \, ds \right)^\frac{p-1}{q+1-p} & \, \int_r^{+\infty} \frac{1}{\psi^{\frac{n-1}{p-1}}} \, ds \le \left[ \int_0^r  \left( \frac{\int_0^s \psi^{n-1} \, \psi' \, dt }{\psi^{n-1}(s)} \right)^{\frac{1}{p-1}} \, ds \right]^\frac{p-1}{q+1-p} \int_r^{+\infty} \frac{1}{\psi^{\frac{n-1}{p-1}}} \, ds \\
& \le  \, \frac{1}{n^{\frac{1}{q+1-p}}} \left( \int_0^r \psi^{\frac{1}{p-1}} \, ds \right)^\frac{p-1}{q+1-p} \int_r^{+\infty} \frac{1}{\psi^{\frac{n-1}{p-1}}} \, ds \\
& \le  \, \frac{r^{\frac{p-1}{q+1-p}}}{n^{\frac{1}{q+1-p}}} \, \psi^{\frac{1}{q+1-p}}(r) \, \int_r^{+\infty} \frac{1}{\psi^{\frac{n-1}{p-1}}} \, ds 
\le  \, \frac{r^{\frac{p-1}{q+1-p}}}{n^{\frac{1}{q+1-p}}} \, \int_r^{+\infty} \frac{1}{\psi^{\frac{n-1}{p-1} -\frac{1}{q+1-p} }} \, ds \\
& \le  \, \frac{r^{\frac{p-1}{q+1-p}}}{n^{\frac{1}{q+1-p}}} \, \int_r^{+\infty} \frac{1}{s^{\frac{n-1}{p-1} -\frac{1}{q+1-p} }} \, ds =  \, \frac{r^{\frac{p}{q+1-p} - \frac{n-p}{p-1} }}{n^{\frac{1}{q+1-p}} \left( \frac{n-1}{p-1} -\frac{1}{q+1-p} -1 \right)}  \, ,
\end{aligned}
$$
and it is readily seen that for $ p \in (1,n) $ and $ q +1 \ge p^\ast $ the power appearing in the last identity is negative. In particular, from \eqref{hp add 3}, \eqref{ee2} and \eqref{ee2-bis} we necessarily deduce that 
$$
\lim_{s \to + \infty} \int_0^s \, \psi^{n-1} \, u^q \, dt = \lim_{s \to + \infty} \int_{0}^s \, \psi^{n-1}(t) \left( \int_0^t  \Theta ^{\frac{1}{p-1}} \, d\tau \right)^{-\frac{(p-1)q}{q+1-p}} dt = +\infty \, .
$$
As a result, we can assert that \eqref{ee2} is actually equivalent to 
\beq\label{ee4}
\lim_{r \to +\infty} \left( \int_0^r  \Theta^{\frac{1}{p-1}} \, ds \right)^\frac{p-1}{q+1-p} \int_r^{+\infty} \left[ \frac{\int_{0}^s \, \psi^{n-1}(t) \left( \int_0^t  \Theta^{\frac{1}{p-1}} \, d\tau \right)^{-\frac{(p-1)q}{q+1-p}} dt }{\psi^{n-1}(s)} \right]^{\frac{1}{p-1}} ds = \, \frac{q+1-p}{p-1} \, .
\eeq
Integration by parts, along with \eqref{ee1}, yields 
$$
\begin{aligned}
 \int_{0}^s \, \psi^{n-1}(t) \left( \int_0^t  \Theta^{\frac{1}{p-1}} \, d\tau \right)^{-\frac{(p-1)q}{q+1-p}} dt & =
 \int_0^s \psi^{n-1} \, dt  \left( \int_0^s  \Theta^{\frac{1}{p-1}} \, dt \right)^{-\frac{(p-1)q}{q+1-p}}  \\
 & \quad + \frac{(p-1)q}{q+1-p} \, \int_0^s \Theta^{\frac{1}{p-1}}(t) \left( \int_0^t \Theta^{\frac{1}{p-1}} \, d\tau \right)^{-\frac{(p-1)q}{q+1-p}-1} \int_0^t \psi^{n-1} \, d\tau \, dt \\
 & \ge \left[ 1 + \frac{(p-1)q \hat{\kappa}}{q+1-p} \right] \int_0^s \psi^{n-1} \, dt  \left( \int_0^s  \Theta^{\frac{1}{p-1}} \, dt \right)^{-\frac{(p-1)q}{q+1-p}} 
\end{aligned}
$$
for all $ s \ge \hat{r} $, whence 
$$
\begin{aligned}
 \int_r^{+\infty} & \left[ \psi^{1-n}(s) \int_{0}^s \, \psi^{n-1}(t) \left( \int_0^t  \Theta^{\frac{1}{p-1}} \, d\tau \right)^{-\frac{(p-1)q}{q+1-p}} dt  \right]^{\frac{1}{p-1}} ds  \\
 & \qquad \ge  \left[ 1 + \frac{(p-1)q\hat{\kappa}}{q+1-p} \right]^{\frac{1}{p-1}} \int_r^{+\infty} \Theta^{\frac{1}{p-1}}(s) \left( \int_0^s  \Theta^{\frac{1}{p-1}} \, dt \right)^{-\frac{q}{q+1-p}}  ds \\
& \qquad =  \, \frac{q+1-p}{p-1} \left[ 1 + \frac{(p-1)q\hat{\kappa}}{q+1-p} \right]^{\frac{1}{p-1}} \left( \int_0^r  \Theta^{\frac{1}{p-1}} \, ds \right)^{-\frac{p-1}{q+1-p}} 
\end{aligned}
$$
for all $r \ge \hat{r}$, which is clearly in contradiction with \eqref{ee4}. 

If \eqref{hp add 1} holds with $ \gamma=1 $, then by L'H\^opital's rule it is readily seen that 
$$
\lim_{r \to +\infty} \frac{\Theta(r)}{r} = \frac{1}{(n-1)\ell + 1} \, ,
$$
whence 
$$
\lim_{r \to +\infty} \frac{ \int_0^r \Theta^{\frac{p}{p-1}} \, \psi^{n-1} \, ds  }{ \psi^{n-1}(r) \, \Theta(r) \int_0^r \Theta^{\frac{1}{p-1}} \, ds } =  \frac{p}{p-1} \, \lim_{r \to +\infty} \frac{ \int_0^r s^{\frac{p}{p-1}} \, \psi^{n-1}(s) \, ds  }{r^{\frac{2p-1}{p-1}}  \psi^{n-1}(r)} = \frac{p}{2p-1+(n-1)(p-1)\ell} > 0 \, ,
$$
where in the last passage we have used again L'H\^opital's rule. 

Finally, given $ \alpha>0 $, let us exhibit a $p$-stochastically complete Cartan-Hadamard model manifold, satisfying \eqref{hp fail exp}, where the limit in \eqref{hp add 3} of the solution $u$ to \eqref{radial pb} that starts from $ u(0)=\alpha $ cannot exist. We will proceed by means of a recursive construction. First of all, note that in view of the asymptotic results established above we know in particular that if $\psi $ complies with 
$$
\lim_{r \to +\infty} \frac{\psi'(r)}{\psi(r)} \in (0,+\infty)
$$
then $u$ satisfies \eqref{hp add 3}. On the contrary, by reasoning similarly to the above disproof of \eqref{hp add 3} under \eqref{hp fail}, it is not difficult to check that if $ \psi $ fulfills
$$
\lim_{r \to +\infty} \frac{\psi(r)}{r} \in [1,+\infty)
$$
then 
\beq\label{asymp-1}
\liminf_{r \to +\infty} \left( \int_0^r  \Theta^{\frac{1}{p-1}} \, ds \right)^\frac{p-1}{q+1-p} u(r) \le \left(\frac{p-1}{q+1-p}  \right)^\frac{p-1}{q+1-p} \left[ 1-\frac{pq}{n(q+1-p)} \right]^{\frac{1}{q+1-p}} .
\eeq
Our strategy strongly relies on this dichotomy. We pick an increasing sequence of radii $ \{ r_k \}_{k \in \N} \subset [0,+\infty)  $ and a corresponding sequence of smooth convex functions $ \{ \psi_k \}_{k \in \N} $, which will be carefully chosen below, such that the global function $\psi$ defined by
\beq\label{psi-recursive}
\psi(r) := \psi_k(r) \qquad \text{for every } r \in [r_k , r_{k+1} ) 
\eeq
gives rise to a $p$-stochastically complete Cartan-Hadamard model manifold that meets our purpose. For the sake of readability, along the $k$-th recursive step we will (implicitly) still let $ \psi $ denote the function that fulfills \eqref{psi-recursive} on the whole $ [r_k,+\infty) $ rather than on $ [r_k,r_{k+1} ) $. 

We start the iteration by taking $ r_0=0 $ and $ \psi_0(r) = r $. In particular, due to \eqref{asymp-1} we can pick $ r_1 \ge 1 $ so large that 
$$
\left( \int_0^{r_1}  \Theta^{\frac{1}{p-1}} \, ds \right)^\frac{p-1}{q+1-p} u(r_1) < \left(\frac{p-1}{q+1-p}  \right)^\frac{p-1}{q+1-p} \left( 1-\frac 1 n \right)^{\frac{1}{q+1-p}}  , \qquad u(r_1) < 1 \, ,
$$ 
whereas the next function $ \psi_1 $ is chosen to be smooth, convex, complying with
$$
\lim_{r \to +\infty} \frac{\psi'_1(r)}{\psi_1(r)} = 2
$$
and gluing to $ \psi_0 $ in such a way that $ \psi $ is also globally smooth and convex. Hence, thanks to Theorem \ref{thm: decay}, the solution constructed so far satisfies \eqref{hp add 3}; we are thus allowed to select $ r_2 \ge r_1 + 1 $ so large that 
$$
\left( \int_0^{r_2}  \Theta^{\frac{1}{p-1}} \, ds \right)^\frac{p-1}{q+1-p} u(r_2) > \left(\frac{p-1}{q+1-p}  \right)^\frac{p-1}{q+1-p} \left( 1-\frac {1} {2n} \right)^{\frac{1}{q+1-p}} \, , \qquad u(r_2) < \frac{1}{2} \, , \qquad  \frac{\psi_1(r_2)}{e^{ r_2}} \ge 1 \, .
$$
The subsequent (recursive) steps of the procedure go as follows. Given $ r_0$, $\ldots r_k $ and $ \psi_0$, $ \ldots \psi_{k-1} $, with $ k \ge 2 $ even, first we choose $ \psi_k $ to be smooth, convex, fulfilling
$$
\lim_{r \to +\infty} \frac{\psi_k(r)}{r} \in (1,+\infty)
$$
and gluing to $ \psi_{k-1} $ in such a way that $ \psi $ is also globally smooth and convex. Due to \eqref{asymp-1} we can then pick $ r_{k+1} \ge r_k + 1 $ so large that 
\beq \label{liminf}
\left( \int_0^{r_{k+1}}  \Theta^{\frac{1}{p-1}} \, ds \right)^\frac{p-1}{q+1-p} u(r_{k+1}) < \left(\frac{p-1}{q+1-p}  \right)^\frac{p-1}{q+1-p} \left( 1-\frac 1 n \right)^{\frac{1}{q+1-p}} , \qquad u(r_{k+1}) < \frac{1}{2^{k}} \, .
\eeq
Similarly, given $ r_0$, $ \ldots r_k $ and $ \psi_0$, $ \ldots \psi_{k-1} $, with $k \ge 3$ odd, we choose $ \psi_{k} $ to be smooth, convex, complying with 
$$
\lim_{r \to +\infty} \frac{\psi'_k(r)}{\psi_k(r)} = 2 k
$$
and gluing to $ \psi_{k-1} $ in such a way that $ \psi $ is globally smooth and convex. The solution constructed so far satisfying \eqref{hp add 3}, we can therefore select $ r_{k+1} \ge r_k + 1 $ so large that 
\beq \label{limsup}
\begin{gathered}
\left( \int_0^{r_{k+1}}  \Theta^{\frac{1}{p-1}} \, ds \right)^\frac{p-1}{q+1-p} u(r_{k+1}) > \left(\frac{p-1}{q+1-p}  \right)^\frac{p-1}{q+1-p} \left( 1-\frac {1} {2n} \right)^{\frac{1}{q+1-p}} ,  \\ u(r_{k+1}) < \frac{1}{2^k} \, , \qquad \frac{\psi_k(r_{k+1})}{e^{k r_{k+1}}} \ge 1 \, .
\end{gathered}
\eeq
By applying iteratively this procedure, we end up having constructed a Cartan-Hadamard model manifold represented by a function $ \psi $ defined as in \eqref{psi-recursive} where, by virtue of \eqref{liminf} and \eqref{limsup}, the solution $u$ to \eqref{radial pb} fulfills
\begin{multline*}
\liminf_{r \to +\infty} \left( \int_0^r  \Theta^{\frac{1}{p-1}} \, ds \right)^\frac{p-1}{q+1-p} u(r)  \le \left(\frac{p-1}{q+1-p}  \right)^\frac{p-1}{q+1-p} \left( 1-\frac 1 n \right)^{\frac{1}{q+1-p}} \\
 <  \left(\frac{p-1}{q+1-p}  \right)^\frac{p-1}{q+1-p} \left( 1-\frac {1} {2n} \right)^{\frac{1}{q+1-p}} 
\le  \, \limsup _{r \to +\infty} \left( \int_0^r  \Theta^{\frac{1}{p-1}} \, ds \right)^\frac{p-1}{q+1-p} u(r),
\end{multline*}
and $u(r) \to 0$ as $r \to +\infty$,
while
$$
\limsup_{k \to \infty} \frac{\psi(r_k)}{e^{\ell r_k}} = + \infty \qquad \forall \ell>0 \, .
$$
The thesis is therefore proved (note that in the light of Theorem \ref{thm: sc vs si}, the fact that $u\to 0$ at infinity ensures $p$-stochastic completeness). Let us point out that, in each step, the choice of $ \psi_k $ depends only upon $ \{ r_i \}_{i \in \{ 0, \ldots k \} } $ and $ \{ \psi_i \}_{i \in \{ 0 , \ldots k-1 \} } $, whereas the choice of $ r_{k+1} $ depends upon the same quantities plus the solution $ u $ constructed on $ [0,+\infty) $ with $ \psi(r) \equiv \psi_{k}(r) $ for all $ r \ge r_k $ (thus only on $ \alpha,q,p,n $), leaving however such a solution unchanged in the interval $ [0,r_k] $. 
\normalcolor
\end{proof}

\subsection{Proof of the main results for \texorpdfstring{$\boldsymbol p$}{}-stochasically \emph{incomplete} manifolds.}

Now we address the case when the Cartan-Hadamard manifold at hand is $p$-stochastically incomplete, namely \eqref{p-si} holds. 

We will first need the following elementary lemma, whose simple proof is omitted.

\begin{lemma}\label{lem tecnico}
For all $\alpha>1$ and $\eps>0$, there exists $C_\eps>0$ such that 
\[
\left(x+y\right)^\alpha \le (1+\eps) \, x^\alpha + C_\eps \, y^\alpha \qquad \forall x,y>0 \, .
\] 
\end{lemma}

\begin{proof}[Proof of Theorem \ref{thm: sc vs si}-($ii$)]
As concerns the existence of infinitely-many radial solutions to \eqref{LE}, and the fact that they all satisfy \eqref{subs:1}, the same observations as in the proof of case ($i$) hold.

Let then $u$ be any solution to \eqref{radial pb} under \eqref{p-si}, and suppose by contradiction that $u(r) \to 0$ as $r \to +\infty$. The case $\psi(r) = r$ does not fulfill \eqref{p-si}, and this necessarily means that $\Mb^n \not \equiv \R^n$. Therefore, Lemma \ref{lem: P strict neg} holds and $P_u(r) \le P_u(\overline{r}) =: -C <0$ for all $ r>\overline{r} $, provided $ \overline{r} $ is large enough (here and below $C$ or analogous constants will not be relabeled). In particular, we infer that for every $r> \overline{r}$
\[
\psi^{n-1}(r) \left|u'(r)\right|^{p-2} u'(r) \, u(r) \le - C \quad \iff \quad u^\frac{1}{p-1}(r)  \, u'(r) \le - C \, \psi^\frac{1-n}{p-1}(r)
\]
where we recall that $u'<0$ and $u>0$.
By integrating on $(r,+\infty)$, since we are supposing that $u \to 0$ at infinity, we thus obtain that
\[
-\frac{p-1}{p} \, u^\frac{p}{p-1}(r) \le - C \int_r^{+\infty}  \psi^\frac{1-n}{p-1} \, ds
\]
for all large $r$ enough, which in turn gives
\beq\label{1si}
\left(\int_r^{+\infty}  \psi^\frac{1-n}{p-1} \, ds \right)^\frac{p-1}{p} \le C \, u(r) \to 0 \qquad \text{as $r \to +\infty$} \, .
\eeq
Now we claim that
\beq\label{2si}
\limsup_{r \to +\infty} \frac{u'(r)}{u(r)} \, \psi^\frac{n-1}{p-1}(r) \int_r^{+\infty} \psi^\frac{1-n}{p-1} \, ds  > - 1 \, .
\eeq
To prove the claim, we argue by contradiction and suppose that the above $\limsup$ is less than or equal to $-1$: that is, for all $\eps>0$ there exists $r_\eps>1/\eps$ such that
\[
\frac{u'(r)}{u(r)} \, \psi^\frac{n-1}{p-1}(r) \int_r^{+\infty} \psi^\frac{1-n}{p-1} \, ds \le  - 1 + \eps \qquad \forall r>r_\eps \, .
\]
By integrating we obtain
\[
\log\left(\frac{u(r)}{u(r_\eps)}\right) \le \left( 1 - \eps \right) \log \left( \frac{\int_{r}^{+\infty} \psi^\frac{1-n}{p-1} \, ds }{\int_{r_\eps}^{+\infty} \psi^\frac{1-n}{p-1} \, ds } \right) \qquad \forall r>r_\eps 
\]
(note that the integrals on the right hand side are finite as $ \psi(r) \ge r $ and $ p<n $). In turn, this implies that there exists $C_\eps>0$ such that
\[
u(r) \le C_\eps \left(\int_{r}^{+\infty} \psi^\frac{1-n}{p-1} \, ds\right)^{1- \eps}  \qquad \forall r>r_\eps \, .
\]
However, by taking $\eps<1-(p-1)/p$, a comparison with \eqref{1si} yields
\[
0<C_\eps \le \left(\int_{r}^{+\infty} \psi^\frac{1-n}{p-1} \, ds \right)^{1- \eps- \frac{p-1}{p}} \to 0 \qquad \text{as } r \to +\infty \, , 
\]
which is a contradiction. This proves the claim \eqref{2si}, and hence there exist $\delta \in (0,1)$ and a sequence $ r_m \to +\infty$ such that
\beq\label{3si}
\frac{u'(r_m)}{u(r_m)} \, \psi^\frac{n-1}{p-1}(r_m) \int_{r_m}^{+\infty} \psi^\frac{1-n}{p-1} \, ds > - 1+\delta \qquad \text{for $m$ large enough} \, .
\eeq
Now we go back to the equation solved by $u$ in \eqref{radial pb}: by integrating it on $(r_m,r)$, and using the fact that $u'<0$, we deduce that
\[
\left(-u'(r)\right)^{p-1} = \frac{\psi^{n-1}(r_m) \left(-u'(r_m)\right)^{p-1}}{\psi^{n-1}(r)} + \psi^{1-n}(r)\int_{r_m}^r \psi^{n-1} \, u^q \, ds \qquad \forall r > r_m \, .
\]
Thus
\beq\label{4si}
-u'(r) = \left[ \frac{\psi^{n-1}(r_m) \left(-u'(r_m)\right)^{p-1}}{\psi^{n-1}(r)} + \psi^{1-n}(r)\int_{r_m}^r \psi^{n-1} \, u^q \, ds \right]^\frac{1}{p-1} \qquad \forall r > r_m \, ,
\eeq
and to proceed further we distinguish between two subcases: $p \ge 2$ or $p<2$. 

If $p \ge 2$ then $1/(p-1) \le 1$, and hence $(x+y)^{1/(p-1)} \le x^{1/(p-1)} + y^{1/(p-1)}$ for all $x,y>0$. Therefore, from \eqref{4si} we infer that
$$
\begin{aligned}
-u'(r) 
\le & \, \frac{\psi^{\frac{n-1}{p-1}} (r_m) \left(-u'(r_m)\right)}{\psi^{\frac{n-1}{p-1}}(r)} + u^\frac{q}{p-1}(r_m) \left( \psi^{1-n}(r)\int_{0}^r \psi^{n-1} \, ds \right)^\frac{1}{p-1}
\end{aligned}
$$
for all $ r>r_m $, where in the last passage we have used the the monotonicity of $u$. Note that the term inside brackets is nothing but the function $\Theta(r)$ defined in \eqref{def-theta}, so that a further integration on $(r_m,r)$ gives the estimate
\beq\label{5si}
\begin{aligned}
u(r)
\ge  \, u(r_m) \left(1 + \frac{u'(r_m) }{u(r_m)} \, \psi^{\frac{n-1}{p-1}} (r_m) \int_{r_m}^{+\infty} \psi^{\frac{1-n}{p-1}} \, ds - u^{\frac{q+1-p}{p-1}}(r_m) \int_{r_m}^{+\infty} \Theta^\frac{1}{p-1} \, ds\right)
\end{aligned}
\eeq 
for all $ r>r_m $. Let us focus on the term inside brackets on the right-hand side. By \eqref{3si}, assumption \eqref{p-si} and the fact that $u(r_m) \to 0$ as $m \to \infty$, we obtain
\[
1 + \frac{u'(r_m) }{u(r_m)} \, \psi^{\frac{n-1}{p-1}} (r_m) \int_{r_m}^{+\infty} \psi^{\frac{1-n}{p-1}} \, ds - u^{\frac{q+1-p}{p-1}}(r_m) \int_{r_m}^{+\infty} \Theta^\frac{1}{p-1} \, ds \ge \delta + o(1) \, ,
\]
where $o(1) \to 0$ as $m \to \infty$. In particular, upon taking $m$ sufficiently large and going back to \eqref{5si}, we end up with the estimate $u(r) \ge \delta/2 u(r_m)$ for every $r>r_m$, 
which is in contradiction with the fact that $u(r) \to 0$ as $r \to +\infty$. This shows that, under assumption \eqref{p-si} and supposing that $p \ge 2$, we necessarily have $u(r) \to \lambda >0$ as $r \to +\infty$.

Let us now consider the case $p<2$. Since $1/(p-1)>1$, it is no more true that $(x+y)^{1/(p-1)} \le x^{1/(p-1)} + y^{1/(p-1)}$ for all $x,y>0$. However, we can exploit Lemma \ref{lem tecnico} with $\eps=\delta/2$, where $\delta$ is defined by \eqref{3si}. Thanks to this choice, identity \eqref{4si} implies that
\[
-u'(r)  \le \left(1+\frac{\delta}2\right) \frac{\psi^{\frac{n-1}{p-1}} (r_m) \left(-u'(r_m)\right)}{\psi^{\frac{n-1}{p-1}}(r)} + C_{\frac \delta 2} \, u^\frac{q}{p-1}(r_m) \left( \psi^{1-n}(r)\int_{0}^r \psi^{n-1} \, ds \right)^\frac{1}{p-1} \qquad \forall r>r_m \, .
\]
Similarly to the case $p \ge 2$, we thus deduce that 
\[
u(r) \ge u(r_m) \left[1 + \left(1+\frac{\delta}2\right) \frac{u'(r_m) }{u(r_m)} \, \psi^{\frac{n-1}{p-1}} (r_m) \int_{r_m}^{+\infty} \psi^{\frac{1-n}{p-1}} \, ds - C_{\frac \delta 2} \, u^{\frac{q+1-p}{p-1}}(r_m) \int_{r_m}^{+\infty} \Theta^\frac{1}{p-1} \, ds \right]
\]
for all $ r>r_m $. By virtue of \eqref{3si}, assumption \eqref{p-si} and using the fact that $u(r_m) \to 0$ as $m \to \infty$, the right-hand side is greater than $u(r_m) \left[1+(1+\delta/2)(-1+\delta) + o(1)\right]$, with $o(1) \to 0$ as $m \to \infty$. As a result, by taking $m$ sufficiently large we end up with the estimate $u(r) \ge \delta/4 u(r_m)$ for every $r>r_m$
which gives again a contradiction. This completes the proof also for $p <2$.
\end{proof}

Now we can proceed with the proof of Theorem \ref{thm: rig} in the $p$-stochastically incomplete case. 

\begin{proof}[Proof of Theorem \ref{thm: rig} under \eqref{p-si}]
Suppose by contradiction that there exists a radial solution $u$ to \eqref{LE}, satisfying \eqref{hp grad}, on a Cartan-Hadamard model manifold complying with \eqref{p-si}. As in the first part of the proof, we infer that $u$ is a (classical) solution to \eqref{radial pb} for some $\alpha>0$. By combining the monotonicity of $u$ with Theorem \ref{thm: sc vs si}, we have that $0<\lambda := \lim_{r \to +\infty} u(r) <u(r)<\alpha$ for all $r>0$. Upon integrating the differential equation in \eqref{radial pb}, we deduce that
\beq \label{eq-integrated}
-\psi^{n-1}(r) \left|u'(r)\right|^{p-2} u'(r) = \int_0^r \psi^{n-1} \, u^q \, ds \quad \implies \quad \left|u'(r)\right|^p = \left(\psi^{1-n}(r) \int_0^r \psi^{n-1} \, u^q \, ds \right)^\frac{p}{p-1}
\eeq
for all $r>0$, where we used again the fact that $u'<0$. Upon multiplying by $\psi^{n-1}$, integrating and exploiting the monotonicity of both $ u $ and $ \psi $, we obtain:
\beq\label{8si}
\begin{aligned}
\int_0^{r} \left|u'\right|^p \psi^{n-1} \, ds & = \int_0^{r}  \left(\psi^{1-n}(s) \int_0^s \psi^{n-1} \, u^q \, dt \right)^\frac{p}{p-1} \psi^{n-1}(s) \, ds \\
& \ge \lambda^{\frac{qp}{p-1}} \int_0^{r}  \psi^\frac{(1-n)p}{p-1}(s) \left[ \frac{p-1}{2p-1} \left(\int_0^s \psi^{n-1} \, dt \right)^\frac{2p-1}{p-1} \right]' ds \\
& \ge \lambda^{\frac{qp}{p-1}} \,  \psi^\frac{(1-n)p}{p-1}(r) \int_0^{r} \left[ \frac{p-1}{2p-1} \left(\int_0^s \psi^{n-1} \, dt \right)^\frac{2p-1}{p-1} \right]' ds \\
& \ge \lambda^{\frac{qp}{p-1}} \, \frac{p-1}{2p-1} \left(\int_0^r \psi^{n-1} \right)^\frac{2p-1}{p-1} \psi^\frac{(1-n)p}{p-1}(r)  =  C \, \frac{f^\frac{2p-1}{p-1}(r)}{\left(f'(r)\right)^\frac{p}{p-1}}
\end{aligned}
\eeq
for all $ r>0 $, where $C>0$ is a constant whose explicit value is immaterial, and $f(r) := \int_0^r \psi^{n-1} \, ds $. We claim that 
\beq\label{7si}
\limsup_{r \to +\infty} \frac{f^\frac{2p-1}{p-1}(r)}{\left(f'(r)\right)^\frac{p}{p-1}} = +\infty \, .
\eeq
If not, then the ratio is bounded, which means that for all $r>1$
\[
f'(r) \ge C \left(f(r)\right)^\frac{2p-1}{p} \quad \text{with $f(1)>0$} \qquad \implies \qquad \text{$f$ blows up at a finite $ r_0>1 $} \, ,
\]
in contradiction with the definition of $f$. This proves that \eqref{7si} holds and, recalling \eqref{8si}, we can finally infer that 
\[
\int_0^{+\infty} \left|u'\right|^p \psi^{n-1} \, ds = +\infty \, ,
\]
which is incompatible with \eqref{hp grad}. That is, also under assumption \eqref{p-si} finite-energy solutions do not exist.
\end{proof}

\begin{remark}\label{rem: radial-non-model}
By reasoning similarly to \cite[Section 2.2]{KM}, it is not difficult to check that if $ \Mb^n $ is not a model manifold but still supports a radial solution to \eqref{LE}, then the latter is in fact a solution to \eqref{radial pb} with $ \psi $ replaced by
$$
\psi_\star(r) := \left[ \frac{\sigma(\partial B_r(o))}{n\, \omega_n} \right]^{\frac{1}{n-1}} \, ,
$$
and such a function falls within the Cartan-Hadamard class. Moreover, we have
$$
\int_{\Mb^n} \left|\nabla u\right|^p d{V} = n \, \omega_n \int_0^{+\infty} \left| u' \right|^p \psi_\star^{n-1} \, dr \, ,
$$ 
so that the proof of Theorem \ref{thm: rig} that we have just carried out  applies to this case as well. 
\end{remark}

\begin{proof}[Proof of Theorem \ref{thm: decay}-($ii$)]
By the definition of limit, for every $\eps>0$ (small enough) there exists $r_\eps>0$ such that $\lambda^q-\eps<u^q(r)< \lambda^q+\eps$ for all $r>r_\eps$. Therefore, in view of \eqref{eq-integrated}, for any such $r$ we have that 
\[
\left[ \psi^{1-n}(r) \left( C_{\eps}  + (\lambda^q-\eps) \int_{r_\eps}^{r} \psi^{n-1} \, ds \right)\right]^\frac{1}{p-1}  <-u'(r) <
\left[ \psi^{1-n}(r) \left( C_{\eps}  + (\lambda^q+\eps) \int_{r_\eps}^{r} \psi^{n-1} \, ds \right)\right]^\frac{1}{p-1} ,
\]
with $C_{\eps} := \int_0^{r_\eps} \psi^{n-1} \, u^q \, ds >0$. A further integration on $(r,+\infty)$ yields, still for $ r>r_\eps $, 
\begin{multline*}
 \,\int_r^{+\infty}\left[ \psi^{1-n}(s) \left( C_{\eps}  + (\lambda^q-\eps) \int_{r_\eps}^{s} \psi^{n-1} \, dt \right)\right]^\frac{1}{p-1} \,ds \\
  <\, u(r) - \lambda 
<  \, \int_r^{+\infty}\left[ \psi^{1-n}(s) \left( C_{\eps}  + (\lambda^q+\eps) \int_{r_\eps}^{s} \psi^{n-1} \, dt \right)\right]^\frac{1}{p-1} \,ds \, .
\end{multline*}
By the arbitrariness of $\eps>0$ and L'H\^opital's rule applied to the integral terms, it is not difficult to obtain the desired asymptotic result.

Finally, in order to prove \eqref{limit-lambda}, it is enough to observe that \eqref{24021} actually holds under both \eqref{p-sc} and \eqref{p-si}, whence the universal bound just follows upon letting $ r \to +\infty $ in such estimate.
\end{proof}

\end{document}